\documentclass[10pt,twocolumn,twoside]{IEEEtran}
\usepackage{amsmath,amssymb,euscript ,yfonts,psfrag,latexsym,dsfont,graphicx,bbm,color,amstext,wasysym,subfig,parskip,pdfsync}
\usepackage{flushend}
\graphicspath{{./},{./figures/}}

\newcounter{plm}
\newtheorem{theorem}[plm]{Theorem}
\newtheorem{lemma}[plm]{Lemma}

\newtheorem{remark}[plm]{Remark}
\newtheorem{proposition}[plm]{Proposition}

\begin{document}
%\begin{frontmatter}
\def\spacingset#1{\def\baselinestretch{#1}\small\normalsize}

\title{
%Reversal of the time direction in stochastic models and its use for data fusion in optimal estimation with missing observations
Optimal estimation with missing observations\\ via balanced time-symmetric stochastic models
\thanks{Research supported by grants from AFOSR, NSF, VR, and the SSF.}}

\author{Tryphon T. Georgiou,~\IEEEmembership{Fellow,~IEEE} and Anders Lindquist,~\IEEEmembership{Life Fellow,~IEEE}
\thanks{T.T.\ Georgiou is with the Department of Electrical \& Computer Engineering,
University of Minnesota, Minneapolis, Minnesota; {email: tryphon@umn.edu} and A.\ Lindquist is with the Department of  Automation and the Department of Mathematics, Shanghai Jiao Tong University, Shanghai, China, and the Center for Industrial and Applied Mathematics and ACCESS Linnaeus Center,
KTH Royal Institute of Technology, Stockholm, Sweden; {email: alq@kth.se}}
} 

\maketitle

\setlength{\parindent}{15pt}
\parskip 6pt
\def\spacingset#1{\def\baselinestretch{#1}\small\normalsize}
\newcommand{\E}{\operatorname{E}}
\newcommand{\bX}{\mathbf X}
\newcommand{\bH}{\mathbf H}
\newcommand{\bA}{\mathbf A}
\newcommand{\bB}{\mathbf B}
\newcommand{\bN}{\mathbf N}

\newcommand{\bHo}{{\stackrel{\circ}{\mathbf H}}}

\newcommand{\bHom}{{\stackrel{\circ\,\,}{{\mathbf H}_{t}}}{\hspace*{-13pt}\phantom{H}}^-}
\newcommand{\bHop}{{\stackrel{\circ\,\,}{{\mathbf H}_{t}}}{\hspace*{-13pt}\phantom{H}}^+}
\newcommand{\bHotwo}{{\stackrel{\circ\,\,\,\,\,}{{\mathbf H}_{t_2}}}{\hspace*{-16pt}\phantom{H}}^-}
\newcommand{\bHomn}{{\stackrel{\circ}{{\mathbf H}}}{\hspace*{-9pt}\phantom{H}}^-}
\newcommand{\bHopn}{{\stackrel{\circ}{{\mathbf H}}}{\hspace*{-9pt}\phantom{H}}^+}

\newcommand{\EbHom}{\E^{{\stackrel{\circ\,\,}{{\mathbf H}_{t}}}{\hspace*{-11pt}\phantom{H}}^-}}
\newcommand{\EbHop}{\E^{{\stackrel{\circ\,\,}{{\mathbf H}_{t}}}{\hspace*{-10pt}\phantom{H}}^+}}

\spacingset{.97}

\begin{abstract} 
We consider data fusion for the purpose of smoothing and interpolation based on observation records with missing data. Stochastic processes are generated by linear stochastic models. The paper begins by drawing a connection between time reversal in stochastic systems and all-pass extensions. A particular normalization (choice of basis) between the two time-directions allows the two to share the same {\em orthonormalized\/} state process and simplifies the mathematics of data fusion. In this framework we derive  symmetric and balanced Mayne-Fraser-like formulas that apply simultaneously to smoothing and interpolation. 
\end{abstract}

%\begin{keyword}
%Stochastic realization theory, time-reversal of stochastic models
%\end{keyword}

%\end{frontmatter}
%===============================================================================
\newcommand{\mR}{{\mathbb R}}
\newcommand{\mZ}{{\mathbb Z}}
\newcommand{\mN}{{\mathbb N}}
\newcommand{\mE}{{\mathbb E}}
\newcommand{\mC}{{\mathbb C}}
\newcommand{\mD}{{\mathbb D}}
\newcommand{\bU}{{\mathbf U}}
\newcommand{\bW}{{\mathbf W}}
\newcommand{\cF}{{\mathcal F}}

\newcommand{\trace}{{\rm trace}}
\newcommand{\rank}{{\rm rank}}
\newcommand{\Real}{{\Re}e\,}
\newcommand{\half}{{\frac12}}

\spacingset{1.1}

\section{Introduction}
Data fusion is the process of integrating different data sets, or statistics, into a more accurate representation for a quantity of interest. A case in point in the context of systems and control is provided by the Mayne-Fraser two-filter formula \cite{Mayne-66,Fraser-P-69}
in which the estimates generated by two different filters are merged into a combined more reliable estimate in fixed-interval smoothing. The purpose of this paper is to develop such a two-filter formula that is universally applicable to smoothing and interpolation based on general records with missing observations.  

In \cite{Badawi-L-P-79,badawi1979mayne} the Mayne-Fraser formula was analyzed in the context of stochastic realization theory
and was shown that it can be formulated in terms a forward and a backward Kalman filter. In a subsequent series of papers, Pavon
 \cite{Pavon-84,Pavon-84a} addressed in a similar manner the hitherto challenging problem of interpolation \cite{karhunen1952interpolation,rozanovstationary,masani1971review,dym2008gaussian}. This latter problem consists of reconstructing missing values of a stochastic process over a given interval. In departure from the earlier statistical literature, \cite{Pavon-84,Pavon-84a} considered a stationary process with rational spectral density and, therefore, reliazable as the output of a linear stochastic system. Interpolation was then cast as seeking an estimate of the state process based on an incomplete observation record. A basic tool in these works is the concept of time-reversal in stochastic systems which has been central in  stochastic realization theory (see, e.g., \cite{Lindquist-P-79,Lindquist-P-85a,Lindquist-P-85,Lindquist-P-91}, \cite{Pavon-84,Pavon-84a}, \cite{Lindquist-Pa-84}, \cite{Michaletzky-B-V-98}, \cite{Michaletzky-F-95}). For a recent overview of smoothing and interpolation theory in the context of stochastic realization theory see \cite[Chapter 15]{LPbook}.
 
In the present paper we are taking this program several steps further. Given intermittent observations of the output of a linear stochastic system over a finite interval, we want to determine the linear least-squares estimate of the state of the system in an arbitrary point in the interior of the interval, which may either be in a subinterval of missing data or in one where observations are available. Hence, this combines smoothing and interpolation over general patterns of available observations. Our main interest is in continuous-time (possibly time-varying) systems. However, the absence of data over subintervals, depending on the information pattern, may necessitate a hybrid approach involving discrete-time filtering steps.

In studying the statistics of a process over an interval, it is natural to decompose the interface between past and future in a time-symmetric manner. This gives rise to systems representations of the process running in either time direction, forward or backward in time. This point was fundamental in early work in stochastic realization; see \cite{LPbook} and references therein. In a different context \cite{georgiou2007caratheodory} a certain duality between the two time-directions in modeling a stochastic process was introduced in order to characterize solutions to moment problems. In this new setting the noise-process was general (not necessarily white), and the correspondence between the driving inputs to the two time-opposite models was shown to be captured by suitable dual all-pass dynamics.

% 
%
%Time reversal of stochastic systems is central in stochastic realization theory (see, e.g., \cite{Lindquist-P-79,
%Lindquist-P-85a,Lindquist-P-85,Lindquist-P-91}, \cite{Pavon-84,Pavon-84a}, \cite{Lindquist-Pa-84}, \cite{Michaletzky-B-V-98}, \cite{Michaletzky-F-95}),
%filtering (see \cite{Lindquist-74}),
%smoothing (see \cite{badawi1979mayne,Badawi-L-P-79}, \cite{ferrante2000minimal}) and system identification.
%The principal construction is to model a stochastic process as the output of a linear system driven by a noise process which is assumed to be white in discrete time, and orthogonal-increment in continuous time. In studying the dependence between past and future of the process, it is natural to decompose the interface between past and future in a time-symmetric manner.
%This gives rise to systems representations of the process running in either time direction, forward or backward in time. 

Here, we begin by combining these two sets of ideas to develop a general framework where
two time-opposite stochastic systems model a given stochastic process. We study the relationship between these systems and the corresponding processes. In particular, we recover as a special case certain results of stochastic realization theory
\cite{Lindquist-P-79}, \cite{Pavon-84,Pavon-84a}, \cite{badawi1979mayne} from the 1970's using a novel procedure.
This theory provides a normalized and balanced version of the forward-backward duality which is essential for our new formulation of the two-filter Mayne-Fraser-like formula uniformly applicable to intervals with or without observations.

The paper is structured as follows. In Section \ref{sec:allpass} we explain how a lifting of state-dynamics into an all-pass system allows direct correspondence between sample-paths of driving generating processes, in opposite time-directions, via causal and anti-causal mappings, respectively. This is most easily understood and explained in discrete-time and hence we begin with that.
In Section \ref{sec:timereversal} we utilize this mechanism in the context of general output processes and, similarly, introduce a pair of time-opposite models. 
These two introductory sections, \ref{sec:allpass} and \ref{sec:timereversal}, deal with stationary models for simplicity and are largely based on \cite{GLifac}. The corresponding generalizations to time-varying systems are given in Section \ref{sec:nonstationary} and in the appendix, in continuous and discrete-time, respectively. In Section \ref{Kalmansec} 
we explain Kalman filtering for problems with missing information in the continuous-time setting. 
In this, we first consider the case where increments of the output process across intervals of no information are unavailable as a simplified preliminary, after which we focus on the central problem where the output process is the object of observation. 
Section \ref{sec:geometryoffusion} deals with the geometry of information fusion.
In Section \ref{sec:twofilters} we present a generalized balanced two-filter formula that applies uniformly over intervals where data is or is not available. We summarize the computational steps of this approach in Section~\ref{sec:recap}. 
Finally, we highlight the use of the two-filter formula with a numerical example given in Section \ref{sec:example} and provide concluding remarks in Section \ref{sec:conclusions}.

\section{State dynamics and all-pass extension}\label{sec:allpass}

In this paper we consider discrete-time as well as continuous-time stochastic linear state-dynamics. We begin by explaining basic ideas in a stationary setting.  In discrete-time systems take the form of a set of difference equations
\begin{align}\label{eq:model_discrete}
x(t+1)=Ax(t)+Bw(t)
\end{align}
where $t\in\mZ$, $A\in\mR^{n\times n}, B\in\mR^{n\times p}$, $A$ has all eigenvalues in the open unit disc $\mD=\{z\mid |z|<1\}$, and $w(t),x(t)$ are (centered) stationary vector-valued stochastic processes with $w(t)$ normalized white noise; i.e., 
\begin{equation}
\label{eq:u}
\E\{w(t)w(s)'\}=I_p\delta_{ts}, 
\end{equation}
where $\E$ denotes mathematical expectation.
 The system of equations is assumed to be reachable, i.e.,
\begin{align}\label{eq:controllability}
\rank\left[ B,\,AB,\,\ldots A^{n-1}B\right]=n.
\end{align}

In continuous-time, state-dynamics take the form of a system of stochastic differential equations
\begin{align}\label{eq:model_continuous}
dx(t)=Ax(t)dt+Bdw(t)
\end{align}
where, here, $x(t)$ is a stationary continuous-time vector-valued stochastic process and $w(t)$ is a vector-valued process with orthogonal increments with the property
\begin{equation}
\label{eq:du_property}
\E\{dwdw'\}=I_pdt,
\end{equation}
where $I_p$ is the $p\times p$ identity matrix. 
Reachability of the pair $(A,B)$ is also assumed throughout
and the condition for this is identical to the one for discrete-time given above (as is well known). In continuous time, stability of the system of equations is equivalent to $A$ having only eigenvalues with negative real part.

In either case, discrete-time or continuous-time, it is possible to define an output equation so that the overall system is all-pass. This is done next. 

\subsection{All-pass extension in discrete-time}\label{Section21}

Consider the discrete-time Lyapunov equation
\begin{align}\label{discreteLyapunov}
P= APA^\prime + BB^\prime.
\end{align}
Since $A$ has all eigenvalues inside the unit disc of the complex plane and \eqref{eq:controllability} holds, \eqref{discreteLyapunov} has as solution a matrix $P$ which is positive definite.
The state transformation
\begin{align}\label{eq:statetransform}
\xi= P^{-\half} x,
\end{align}
and
\begin{align}
F=P^{-\half} A P^{\half},\;G=P^{-\half} B,\label{eq:similarity}
\end{align}
brings \eqref{eq:model_discrete} into
\begin{align}\label{eq:model_discrete2}
\xi(t+1)=F\xi(t)+Gw(t).
\end{align}

For this new system, the corresponding Lyapunov equation $X=FXF^\prime+GG^\prime$ has $I_n$ as solution,
where $I_n$ denotes the $(n\times n)$ identity matrix. This fact, namely, that
\begin{align}\label{eq:FF'+GG'}
I_n=FF^\prime + GG^\prime
\end{align}
implies that
this $[F,G]$ can be embedded as part of an orthogonal matrix
\begin{align}\label{eq:U}
U=\left[\begin{array}{cc} F& G\\ H& J\end{array}\right],
\end{align}
i.e., a matrix such that $UU^\prime=U^\prime U= I_{n+p}$. 

Define the transfer function
\begin{align}\label{eq:tfdiscrete}
\bU(z):=H(zI_n-F)^{-1}G+J
\end{align}
corresponding to
\begin{subequations}\label{xisystem}
\begin{align}\label{xisystemfirst}
\xi(t+1)&=F\xi(t)+Gw(t)\\
\bar w(t)&=H\xi(t)+Jw(t).
\end{align}
\end{subequations}
This is also the transfer function of
\begin{subequations}\label{xsystem}
\begin{align}
x(t+1)&=Ax(t)+Bw(t)\\
\bar w(t)&=\bar{B}' x(t)+Jw(t),
\end{align}
\end{subequations}
where $\bar{B}:=P^{-\half}H'$, since the two systems are related by a similarity transformation. Hence,
\begin{align}\label{eq:tfdiscrete2}
\bU(z)=\bar{B}'(zI_n-A)^{-1}B+J.
\end{align}
We claim that $\bU(z)$
is a stable all-pass transfer function (with respect to the unit disc), i.e.,
that $\bU(z)$ is a transfer function of a stable system and that
\begin{align}\label{eq:discreteallpass}
\bU(z)\bU(z^{-1})^\prime=\bU(z^{-1})^\prime \bU(z)=I_p.
\end{align}

The latter claim is immediate after we observe that, since $U^\prime U=I_{n+p}$,
\[
U^\prime\left[\begin{array}{c} \xi(t+1)\\ \bar w(t)\end{array}\right]=\left[\begin{array}{c} \xi(t)\\ w(t)\end{array}\right],
\]
and hence,
\begin{subequations}\label{inversexisystem}
\begin{align}
\xi(t)&=F^\prime\xi(t+1)+H^\prime \bar w(t)\\
w(t)&=G^\prime \xi(t+1)+J^\prime \bar w(t)
\end{align}
\end{subequations}
or, equivalently,
\begin{subequations}\label{inversexsystemprel}
\begin{align}
x(t)&=PA^\prime P^{-1}x(t+1)+P^{\half}H^\prime \bar w(t)\\
w(t)&=B^\prime P^{-1} x(t+1)+J^\prime \bar w(t).
\end{align}
\end{subequations}
Setting 
\begin{equation}
\label{xbar}
\bar{x}(t):=P^{-1} x(t+1),
\end{equation}
\eqref{inversexsystemprel} can be written
\begin{subequations}\label{inversexsystem}
\begin{align}\label{inversexsystema}
\bar{x}(t-1)&=A^\prime\bar{x}(t)+\bar B \bar w(t)\\
w(t)&=B^\prime \bar{x}(t)+J^\prime \bar w(t)\label{inversexsystemb}
\end{align}
\end{subequations}
with transfer function
\begin{align}\label{Ustardefinition}
\bU(z)^*=B^\prime(z^{-1}I_n-A^\prime)^{-1}\bar{B}+J^\prime.
\end{align}
Either of the above systems inverts the dynamical relation $w\to \bar w$ (in \eqref{xsystem} or \eqref{xisystem}).

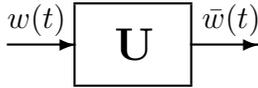
\begin{figure}[h]
\begin{center}
\setlength{\unitlength}{.006in}%{0.012500in}
\parbox{304\unitlength}
{\begin{picture}(100,90)
\thicklines
\put(26,66){\makebox(0,0){{\large $w(t)$}}}
\put(0,45){\vector(1,0){60}}
\put(60,10){\framebox(100,70){{\LARGE
\hspace*{7pt}$\bU$ 
%$\begin{array}{c}{\rm forward}\\{\rm dynamics}\end{array}$
}}}
\put(196,66){\makebox(0,0){{\large $\bar w(t)$}}}
\put(160,45){\vector(1,0){60}}
\end{picture}
}
\end{center}
\caption{Realization \eqref{xsystem} in the forward time-direction.}
\label{forwardfigure}
\end{figure}

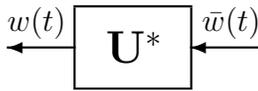
\begin{figure}[h]
\begin{center}
\setlength{\unitlength}{.006in}%{0.012500in}
\parbox{304\unitlength}
{\begin{picture}(100,90)
\thicklines
\put(26,66){\makebox(0,0){{\large $w(t)$}}}
\put(60,45){\vector(-1,0){60}}
\put(60,10){\framebox(100,70){{\LARGE
\hspace*{7pt}$\bU^*$ 
%$\begin{array}{c}{\rm backward}\\{\rm dynamics}\end{array}$
}}}
\put(196,66){\makebox(0,0){{\large $\bar w(t)$}}}
\put(220,45){\vector(-1,0){60}}
\end{picture}
}
\end{center}
\caption{Realization \eqref{inversexsystem} in the backward time-direction.}
\label{backwardfigure}
\end{figure}

An algebraic proof of \eqref{eq:discreteallpass} is also quite immediate. In fact,
\begin{align*}
&\bU(z)\bU(z^{-1})'\\ =& \left[H(zI_n-F)^{-1}G+J\right]\left[H(z^{-1}I_n-F)^{-1}G+J\right]'  \\
   =&H(zI_n-F)^{-1}GG'(z^{-1}I_n-F')^{-1}H' + JJ'\\
   &+H(zI_n-F)^{-1}GJ' +JG'(z^{-1}I_n-F')^{-1}H
\end{align*}
Now, using the identity
\begin{align*}
\label{}
I_n-FF'   &=(zI_n-F)(z^{-1}I_n-F')  \\
              &+ (zI_n-F)F' +F(z^{-1}I_n-F'),
\end{align*}
\eqref{eq:FF'+GG'} and $GJ'=-FH'$, obtained from $UU'=I_{n+p}$, this yields
\begin{displaymath}
\bU(z)\bU(z^{-1})'=HH'+JJ'=I_{n+p},
\end{displaymath}
as claimed.

\subsection{All-pass extension in continuous-time}\label{Section22}

Consider the continuous-time Lyapunov equation
\begin{align}\label{continuousLyapunov}
AP+PA^\prime + BB^\prime =0.
\end{align}
Since $A$ has all its eigenvalues in the left half of the complex plane and since \eqref{eq:controllability} holds, \eqref{continuousLyapunov} has as solution a positive definite matrix $P$.
Once again, applying (\ref{eq:statetransform}-\ref{eq:similarity}),
the system in \eqref{eq:model_continuous}
becomes
\begin{subequations}
\begin{align}\label{eq:model_continuous2}
d\xi(t)=F\xi(t)dt+Gdw(t).
\end{align}
We now seek a completion by adding an output equation
\begin{align}\label{eq:model_output2}
d\bar w(t)=H\xi(t)dt+Jdw(t)
\end{align}
\end{subequations}
so that the transfer function
\begin{align}
\bU(s):=H(sI_n-F)^{-1}G + J
\end{align}
is all-pass
(with respect to the imaginary axis), i.e.,
\begin{align}
\bU(s)\bU(-s)^\prime = \bU(-s)^\prime\bU(s)=I_p.
\end{align}

For this new system, the corresponding Lyapunov equation has as solution the identity matrix and hence,
\begin{align}\label{continuousLyapunov2}
F+F^\prime + GG^\prime =0.
\end{align}
Utilizing this relationship we
note that
\begin{align*}
&(sI_n-F)^{-1}GG^\prime (-sI_n-F^\prime)^{-1}\\
&=(sI_n-F)^{-1}(sI_n-F -sI_n-F^\prime)(-sI_n-F^\prime)^{-1}\\
&=(sI_n-F)^{-1}+(-sI_n-F^\prime)^{-1},
\end{align*}
and we calculate that
\begin{align*}
&\bU(s)\bU(-s)^\prime\\
&= (H (sI_n-F)^{-1}G+J)(G^\prime(-sI_n-F^\prime)^{-1}H^\prime +J^\prime)\\
&=JJ^\prime +H(sI_n-F)^{-1}(GJ^\prime+H^\prime)\\
&\hspace*{1cm}(JG^\prime + H)(-sI_n-F^\prime)^{-1}H^\prime.
\end{align*}
For the product to equal the identity,
\begin{align*}
&JJ^\prime=I_p\\
&H=-JG^\prime.
\end{align*}
Thus, we may take
\begin{align*}
&J=I_p\\
&H=-G^\prime,
\end{align*}
and the forward dynamics
\begin{subequations}\label{eq:continuousforward}
\begin{align}\label{eq:model_continuous2a}
d\xi(t)=F\xi(t)dt+Gdw(t)\phantom{x}\\
\label{eq:model_output2a}
d\bar w(t)=-G^\prime\xi(t)dt+dw(t).
\end{align}
\end{subequations}
Substituting $F=-F^\prime-GG^\prime$ from \eqref{continuousLyapunov2} into \eqref{eq:model_continuous2a} we obtain the reverse-time dynamics
\begin{subequations}\label{xibackward}
\begin{align}\label{eq:model_continuous2a_reverse}
d\xi(t)=-F^\prime\xi(t)dt+Gd\bar w(t)\\
\label{eq:model_output2a_reverse}
dw(t)=G^\prime\xi(t)dt+d\bar w(t).\phantom{xll}
\end{align}
\end{subequations}
Now defining 
\begin{equation}
\label{xbarcont}
\bar{x}(t):=P^{-1}x(t)
\end{equation}
and using \eqref{eq:statetransform} and \eqref{eq:similarity}, \eqref{xibackward} becomes
\begin{subequations}\label{cbackward}
\begin{align}
d\bar{x}(t)=-A^\prime\bar{x}(t)dt+\bar{B}d\bar w(t)\\
dw(t)=B^\prime\bar{x}(t)dt+d\bar w(t),\phantom{xx}\label{ubar2ucont}
\end{align}
\end{subequations}
with transfer function
\begin{align}\label{Ustardefinition_continuous}
\bU(s)^*=I_p+B^\prime(sI_n+A^\prime)^{-1}\bar{B},
\end{align}
where 
\begin{equation}
\label{Bbarcont}
\bar{B}:=P^{-1}B.
\end{equation}
Furthermore, the forward dynamics \eqref{eq:continuousforward} can be expressed in the form
\begin{subequations}\label{cforward}
\begin{align}
dx(t)=Ax(t)dt+Bdw(t)\\
d\bar w(t)=\bar B^\prime x(t)dt+d w(t)\phantom{b}\label{cforwardb}
\end{align}
\end{subequations}
with transfer function
\begin{align}\label{Udefinition_continuous}
\bU(s)=I_p-\bar B^\prime(sI_n-A)^{-1}B.
\end{align}

\section{Time-reversal of stationary linear stochastic systems}\label{sec:timereversal}

The development so far allows us to draw a connection between two linear stochastic systems having the same output and driven by a pair of arbitrary, but dual, stationary processes $w(t)$ and $\bar w(t)$, one evolving forward in time and one evolving backward in time. When one of these two processes is white noise (or, orthogonal increment process, in continuous-time),
then so is the other. For this special case we recover
results of \cite{Lindquist-P-79} and \cite{Pavon-84,Pavon-84a}
in stochastic realization theory.

\subsection{Time-reversal of discrete-time stochastic systems}

Consider a stochastic linear system
\begin{subequations}\label{dsystforward}
\begin{align}
&x(t+1)=Ax(t)+Bw(t) \label{dsystforwarda}\\
&\phantom{xxll}y(t)=Cx(t)+Dw(t) \label{dsystforwardb}
\end{align}
\end{subequations}
with an $m$-dimensional output process $y$, and $x,u,A,B$ are defined as in Section \ref{Section21}.
All processes are stationary and the system can be thought as evolving forward in time from the remote past ($t=-\infty$). 

To formalize this, we introduce some notation. 
Let $\bH$ be the Hilbert space spanned by  $\{w_k(t);\, t\in\mathbb{Z},\, k=1,2,\dots,n\}$, endowed with the inner product $\langle\lambda,\mu\rangle =\E\{\lambda\mu\}$, and let $\bH_t^-(w)$ and $\bH_t^+(w)$ be the (closed) subspaces spanned by  $\{w_k(s);\, s\leq t-1,\, k=1,\dots,m\}$ and $\{w_k(s);\, s\geq t,\, k=1,\dots,m\}$, respectively. Define $\bH_t^-(y)$ and $\bH_t^+(y)$ accordingly  in terms of the output process process $y$. Then the stochastic system \eqref{dsystforward} evolves forward in time in the sense that 
\begin{equation}
\label{forward}
\bH_t^-(z)\subset \bH_t^-(w)\perp\bH_t^+(w) ,
\end{equation} 
where $\bA\perp\bB$ means that elements of the subspaces $\bA$ and $\bB$  are mutually orthogonal, and where $\bH_t^-(z)$ is formed as above in terms of
\begin{displaymath}
z(t)=\begin{bmatrix} x(t+1)\\y(t)\end{bmatrix};
\end{displaymath}
see \cite[Chapter 6]{LPbook} for more details. 

Next we construct a stochastic system
\begin{subequations}\label{dsystbackward}
\begin{align}
&\bar{x}(t-1)=A'\bar{x}(t)+\bar{B}\bar{w}(t) \label{dsystbackwarda}\\
&\phantom{xxx}y(t)=\bar{C}\bar{x}(t)+ \bar{D}\bar{w}(t),\label{dsystbackwardb}
\end{align}
\end{subequations}
which evolves backward in time from the remote future ($t=\infty$) in the sense that
the processes $\bar x,x,\bar w,w$ relate as in the previous section. More specifically, as shown in Section \ref{Section21}, $\bH^-(\bar{w})\subset\bH^-(w)$ and 
$\bH^+(w)\subset\bH^+(\bar{w})$ for all $t$, as examplified in Figures \ref{forwardfigure} and \ref{backwardfigure}.

In fact, the all-pass extension \eqref{xsystem} of \eqref{dsystforwarda} yields 
\begin{equation}
\label{w2wbar}
\bar{w}(t)  =  \bar{B}'x(t)+Jw(t)
\end{equation}
It follows from \eqref{inversexsystemb} that \eqref{w2wbar} can be inverted to yield
\begin{equation}
\label{wbar2w}
w(t)=B'\bar{x}(t)+J'\bar{w}(t), 
\end{equation}
where $\bar{x}(t)=P^{-1}x(t+1)$, and that we have the reverse-time recursion
\begin{subequations}\label{backwardagain}
\begin{equation}
\label{xtminus1}
\bar{x}(t-1)=A'\bar{x}(t)+\bar{B}\bar{w}(t).
\end{equation}
Then inserting \eqref{wbar2w} and 
\begin{displaymath}
x(t)=P\bar{x}(t-1)=PA'\bar{x}(t)+P\bar{B}\bar{w}(t)
\end{displaymath}
into \eqref{dsystforwardb}, we obtain 
\begin{equation}
y(t)=\bar{C}\bar{x}(t)+ \bar{D}\bar{w}(t),
\end{equation}
\end{subequations}
where $\bar{D}:=CP\bar{B}+DJ'$ and
\begin{equation}
\label{Cbar}
\bar{C}:= CPA' +DB' .
\end{equation}
Then, \eqref{backwardagain} is precisely what we wanted to establish.

The white noise $w$ is normalized in the sense of \eqref{eq:u}. Since $\mathbf{U}$, given by \eqref{eq:tfdiscrete2}, is all-pass, $\bar{w}$ is also a normalized white noise process, i.e.,
\[
\E\{\bar{w}(t)\bar{w}(s)'\}=I_p\delta_{t-s}.
\]
From the reverse-time recursion \eqref{dsystbackwarda}
\[
\bar x(t)=\sum_{k=t+1}^\infty (A')^{k-(t+1)}\bar B \bar w(k).
\]
Since, $\bar w$ is a white noise process, $\E\{\bar x(t)\bar w(s)'\}=0$ for all $s\leq t$.
Consequently, \eqref{dsystbackward} is a backward stochastic realization in the sense defined above.

Moreover, the transfer functions
\begin{align}\bW(z)=C(zI_n-A)^{-1}B+D
\end{align}
of \eqref{dsystforward}
and
\begin{align}
\bar\bW(z)=\bar C(z^{-1}I_n-A')^{-1}\bar B+\bar D
\end{align}
of \eqref{dsystbackward} satisfy
\begin{align}
\bW(z)=\bar\bW(z)\bU(z).
\end{align}
In the context of stochastic realization theory,  $\bU(z)$ is called {\em structural function} (\cite{Lindquist-P-85,Lindquist-P-91}).

\begin{figure}[h]
\begin{center}
\setlength{\unitlength}{.006in}%{0.012500in}
\parbox{304\unitlength}
{\begin{picture}(100,90)
\thicklines
\put(26,66){\makebox(0,0){{\large $w(t)$}}}
\put(0,45){\vector(1,0){60}}
\put(60,10){\framebox(100,70){{\LARGE
\hspace*{7pt}$\bW$ 
%$\begin{array}{c}{\rm forward}\\{\rm dynamics}\end{array}$
}}}
\put(196,66){\makebox(0,0){{\large $y(t)$}}}
\put(160,45){\vector(1,0){60}}
\end{picture}
}
\end{center}
\caption{The forward stochastic system \eqref{dsystforward}.}
\label{forwardfigure2}
\end{figure}
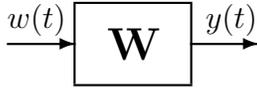
\begin{figure}[h]
\begin{center}
\setlength{\unitlength}{.006in}%{0.012500in}
\parbox{304\unitlength}
{\begin{picture}(100,90)
\thicklines
\put(26,66){\makebox(0,0){{\large $y(t)$}}}
\put(60,45){\vector(-1,0){60}}
\put(60,10){\framebox(100,70){{\LARGE
\hspace*{7pt}$\bar\bW$ 
%$\begin{array}{c}{\rm backward}\\{\rm dynamics}\end{array}$
}}}
\put(196,66){\makebox(0,0){{\large $\bar w(t)$}}}
\put(220,45){\vector(-1,0){60}}
\end{picture}
}
\end{center}
\caption{The backward stochastic system \eqref{dsystbackward}}
\label{backwardfigure2}
\end{figure}
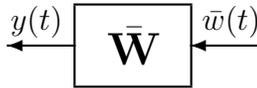

\subsection{Time-reversal of continuous-time stochastic systems}

We now turn to the continuous-time case. Let 
\begin{subequations}\label{csystforward}
\begin{align}
&dx=Axdt+Bdw \label{csystforwarda}\\
&dy=Cxdt+Ddw  \label{csystforwardb}
\end{align}
\end{subequations}
be a stochastic system with $x,w,A,B$ as in Section \ref{Section22}, evolving forward in time from the remote past ($t=-\infty$). Now let 
$\bH$ be the Hilbert space spanned by  the {\em increments\/} of the components of $w$ on the real line $\mathbb{R}$, endowed with the same inner product as above, and let $\bH_t^-(du)$ and $\bH_t^+(du)$ be the (closed) subspaces spanned by the increments of the components of $U$ on $(-\infty,t]$ and $[t,\infty)$, respectively.  Define  $\bH_t^-(dy)$ and $\bH_t^+(dy)$ accordingly in terms of the output process $y$. 
All processes have stationary increments and the stochastic system \eqref{csystforward} evolves forward in time in the sense that 
\begin{equation}
\label{contforward}
\bH_t^-(dz)\subset \bH_t^-(dw)\perp\bH_t^+(dw) ,
\end{equation} 
where $\bH_t^-(dz)$ is formed in terms of
\begin{equation}
\label{eq:xy2z}
z(t)=\begin{bmatrix} x(t)\\y(t)\end{bmatrix}.
\end{equation}

The all-pass extension of Section \ref{Section22} yields
\begin{equation}
d\bar{w}=dw -\bar{B}'xdt
\end{equation}
as well as the reverse-time relation
\begin{subequations}\label{cbackward2}
\begin{align}
d\bar{x}=-A^\prime\bar{x}dt+\bar{B}d\bar w\\
dw=B^\prime\bar{x}dt+d\bar w,\phantom{eq:xy2z}\label{ubar2ucont2}
\end{align}
\end{subequations}
where $\bar{x}(t)=P^{-1}x(t)$. Inserting \eqref{ubar2ucont2} into 
\begin{displaymath}
dy=CP\bar{x}dt+Ddw
\end{displaymath}
yields
\begin{displaymath}
dy=\bar{C}\bar{x}dt+Dd\bar{w},
\end{displaymath}
where
\begin{equation}
\bar{C}=CP+DB'.
\end{equation}
Thus, the reverse-time system is
\begin{subequations}\label{csystbackward}
\begin{align}
&d\bar{x}=-A'\bar{x}dt+\bar{B}d\bar{w} \label{csystbackwarda}\\
&dy=\bar{C}\bar{x}dt+Dd\bar{w}. \label{csystbackwardb}
\end{align}
\end{subequations}
From this, we deduce that the system \eqref{csystforward} has the backward property
\begin{equation}
\label{contbackward}
\bH_t^+(d\bar z)\subset \bH_t^+(d\bar w)\perp \bH_t^-(d\bar w),
\end{equation} 
where $\bH_t^+(d\bar z)$ is formed as above in terms of
\begin{displaymath}
\bar z(t)=\begin{bmatrix} \bar x(t)\\y(t)\end{bmatrix}.
\end{displaymath}
We also note that the transfer function
\[
\bW(s)=C(sI_n-A)^{-1}B+D
\]
of \eqref{csystforward} and the transfer function
\[
\bar\bW(s)=\bar C(sI_n+A')^{-1}\bar B+D
\]
of \eqref{csystbackward} also satisfy
\[
\bW(s)=\bar\bW(s)\bU(s)
\]
as in discrete-time.

Note that the orthogonal-increment process $w$ is  normalized in the sense of \eqref{eq:du_property}.
Since $\bU(s)$ is all-pass,  
\begin{equation}
d\bar{w}=du -\bar{B}'xdt
\end{equation}
also defines a stationary orthogonal-increment process $\bar w$ such that 
\[
\{ d\bar{w}(t)d\bar{w}(t)'\}=I_pdt. 
\]
It remains to show that \eqref{csystbackward}
is a backward stochastic realization, that is, at each time $t$ the past increments of $\bar w$ are orthogonal to $\bar x(t)$.
But this follows from the fact that
\[
\bar x(t)=\int_t^\infty e^{-A' (t-s)}\bar Bd\bar w(s)
\]
and $\bar w$ has orthogonal increments.

\section{Time reversal of non-stationary stochastic systems}\label{sec:nonstationary}

In a similar manner non-stationary stochastic systems admit unitary extensions which in turn allows us to construct dual time-reversed stochastic models that share the same state process. 
The case of discrete-time dynamics is documented in the appendix, whereas the continuous-time counterpart is explained next as prelude to smoothing and interpolation that will follow.

\subsection{Unitary extension} 

The  covariance matrix function $P(t):=\E\{x(t)x(t)'\}$ of the time-varying state representation
\begin{equation}\label{eq:xdiffeq}
dx=A(t)x(t)dt +B(t)dw, \quad x(0)=x_0
\end{equation}
with $x_0$ a zero-mean stochastic vector with covariance matrix $P_0=\E\{x_0x_0'\}$,
satisfies the matrix-valued differential equation
\begin{equation}
\label{eq:Lyapunovdifferentialeq}
\dot{P}(t)=A(t)P(t)+P(t)A(t)'+B(t)B(t)'  
\end{equation}
with $P(0)=P_0$.
Throughout we assume total reachability \cite[Section 15.2]{LPbook}, and therefore $P(t)>0$ for all $t>0$.  

A unitary extension of  \eqref{eq:xdiffeq} is somewhat more complicated than in the discrete time case. 
In fact, differentiating 
\begin{align}\label{eq:cont_statetransform_nonstationary}
\xi(t)= P(t)^{-\half} x(t)
\end{align}
we obtain
\begin{equation}
\label{eq:xicont}
d\xi =F(t)\xi(t)dt +G(t)dw,
\end{equation}
where 
\begin{subequations}\label{eq:AB2FG_cont}
\begin{align}
F(t)&=P(t)^{-\half} A(t) P(t)^{\half}+R(t),\\G(t)&=P(t)^{-\half} B(t)
\end{align}
\end{subequations}
with 
\begin{equation}
\label{eq:R}
R(t)=\left[\frac{d\phantom{t}}{dt}P(t)^{-\half}\right]P(t)^{\half}.
\end{equation}
In fact, 
\begin{equation}
\label{ }
d\xi=P(t)^{-\half}dx+R(t)\xi(t)dt.
\end{equation}
Differentiating $P(t)^{-\half}P(t)P(t)^{-\half}=I_n$, we obtain
\begin{displaymath}
P(t)^{-\half}\dot{P}P(t)^{-\half}=-R(t)-R(t)', 
\end{displaymath}
and hence the \eqref{eq:Lyapunovdifferentialeq} yields
\begin{equation}
\label{eq:F+F'+GG'=0}
F(t)+F(t)'+G(t)G(t)'=0.
\end{equation}
Using \eqref{eq:F+F'+GG'=0} to eliminate $F$ in \eqref{eq:xicont}, we obtain
\begin{equation}
\label{eq:xicontbackward}
d\xi =-F(t)'\xi(t)dt +G(t)d\bar{w},
\end{equation}
where
\begin{equation}
\label{eq:dubar}
d\bar{w}=dw-G(t)'\xi(t)dt,
\end{equation}
which can also be written
\begin{equation}
\label{eq:dubaralt}
d\bar{w}=dw-\bar{B}(t)'x(t)dt,
\end{equation}
where $\bar{B}(t):=P(t)^{-1}B(t)$. 

\begin{proposition}\label{prop:dubar_properties}
A process $\bar{w}$ satisfying \eqref{eq:dubar}  has orthogonal increments with the normalized property \eqref{eq:du_property}. Moreover, 
\begin{equation}
\label{eq:dubar_xi_orthogonality}
\E\{[\bar{w}(t)-\bar{w}(s)]\xi(t)'\}=0
\end{equation}
for all $s\leq t$. 
\end{proposition}

\begin{proof}
As is well-known, the solution of \eqref{eq:xicont} can be written in the form
\begin{equation}
\label{eq:xisolution}
\xi(t)=\Phi(t,s)\xi(s) +\int_s^t\Phi(t,\tau)G(\tau)dw,
\end{equation}
where $\Phi(t,s)$ is the transition matrix with the property
\begin{subequations}\label{eq:transition}
\begin{align}
   \frac{\partial\Phi}{\partial t}(t,s)&=F(t)\Phi(t,s), \quad \Phi(s,s)=I_n   \\
   \frac{\partial\Phi}{\partial s}(t,s)&=-\Phi(t,s)F(s), \quad \Phi(t,t)=I_n   
\end{align}
\end{subequations}
Let $s\leq t$. Then, in view of \eqref{eq:dubar},  a straight-forward calculation yields
\begin{align}
\label{eq:ubar(t)-ubar(s)}
   \bar{w}(t)-\bar{w}(s)&=w(t)-w(s)\notag\\&-M(t,s)\xi(s)-\int_s^t M(t,\tau)G(\tau)dw, 
\end{align}
where 
\begin{equation}
\label{eq:M}
M(t,s)=\int_s^t G(\tau)'\Phi(\tau,s)d\tau. 
\end{equation}
Therefore, 
\begin{displaymath}
\E\{[\bar{w}(t)-\bar{w}(s)][\bar{w}(t)-\bar{w}(s))'\}=I_p (t-s) +\Delta(t,s),
\end{displaymath}
where
\begin{align*}
  \Delta(t,s) &= M(t,s)M(t,s)' + \int_s^t M(t,\tau)G(\tau)G(\tau)'M(t,\tau)'d\tau  \\
    &  -  \int_s^t \left[M(t,\tau)G(\tau)+G(\tau)'M(t,\tau)'\right]d\tau.
\end{align*}
However, $ \Delta(t,s)$ is identically zero. To see this, first note that 
\begin{equation}
\label{eq:Mdiffeq}
\frac{\partial M}{\partial s}(t,s)=-M(t,s)F(s)-G(s)'.
\end{equation}
Then, in view of \eqref{eq:F+F'+GG'=0}, a simple calculation shows that 
\begin{displaymath}
\frac{\partial\Delta}{\partial s}(t,s)\equiv 0.
\end{displaymath}
Since $\Delta(t,t)=0$, the assertion follows. Hence the incremental covariance is normalized. 

Next, we show that $\bar{w}(t)$ has orthogonal increments. To this end, choose arbitrary times $s\leq t\leq  a\leq b$ on the interval $[0,T]$, where we choose $a$ and $b$ fixed, and show that 
\begin{displaymath}
Q(t,s):=\E\{[\bar{w}(b)-\bar{w}(a)][\bar{w}(t)-\bar{w}(s))'\}
\end{displaymath}
is identically zero for all $s\leq t$. Using \eqref{eq:ubar(t)-ubar(s)} and 
\begin{align*}
  \bar{w}(b)-\bar{w}(a)  &= w(b)-w(s) -M(b,a)\Phi(a,s) \xi(s) \\
    &  -M(b,a)\int_s^b\Phi(a,\tau)G(\tau)dw -\int_a^b M(b,\tau)dw
\end{align*}
computed analogously, we obtain
\begin{align*}
  Q(t,s)   =M(b,a)\left[\Phi(a,s)M(t,s)' -\int_s^b\Phi(a,\tau)G(\tau)d\tau \right.&  \\
      + \left.\int_s^b\Phi(a,\tau)G(\tau)G(\tau)'M(t,\tau)d\tau\right] .&
\end{align*}
Then, again using \eqref{eq:F+F'+GG'=0}, we see that 
\begin{displaymath}
\frac{\partial M}{\partial s}(t,s)\equiv 0,
\end{displaymath}
so, since $Q(t,t)=0$, we see that $Q(t,s)$ is identically zero, establishing that $\bar{w}(t)$ has orthogonal increments. 

Finally, we use the same trick to show \eqref{eq:dubar_xi_orthogonality}.  In fact, for $s\leq t$, \eqref{eq:xisolution} and \eqref{eq:ubar(t)-ubar(s)} yield
\begin{align*}
&\E\{[\bar{w}(t)-\bar{w}(s))\xi(t)'\}=-M(t,s)\Phi(t,s)'\\ &+\int_s^t  G(\tau)'\Phi(t,\tau)'d\tau -\int_s^t M(t,\tau)G(\tau)G(\tau)')\Phi(t,\tau)'d\tau, 
\end{align*}
the partial derivative of which with respect to $s$ is identical zero; this is seen by again using \eqref{eq:F+F'+GG'=0}.  Therefore, since  \eqref{eq:dubar_xi_orthogonality} is zero for $s=t$, it is identical zero for all $s\leq t$, as claimed. This concludes the proof of Proposition~\ref{prop:dubar_properties}.
\end{proof}

Consequently, \eqref{eq:xicont} and \eqref{eq:dubaralt} form a forward unitary system
\begin{subequations}\label{xsystem_nonstationarycont}
\begin{align}
dx&=A(t)x(t)dt+B(t)dw\\
d\bar w&=dw -\bar{B}(t)'x(t)dt,
\end{align}
\end{subequations}
The corresponding backward unitary system is obtained through the transformation 
\begin{equation}
\label{ }
\bar{x}(t)=P(t)^{\-\half}\xi(t),
\end{equation}
which yields 
\begin{equation}
\label{ }
d\bar{x}=P(t)^{-\half}d\xi +R(t)\xi(t)dt. 
\end{equation}
This together with \eqref{eq:xicontbackward} and \eqref{eq:dubar} yields
\begin{subequations}\label{xbarsystem_nonstationarycont}
\begin{align}
d\bar{x}&=-A(t)'\bar{x}(t)dt+\bar{B}(t)d\bar{w}\label{xbarsystem_nonstationaryconta}\\
dw&=B(t)'\bar{x}(t)dt+d\bar{w},\label{xbarsystem_nonstationarycontb}
\end{align}
\end{subequations}

\subsection{Time reversal in continuous-time systems} 

Next we derive the backward stochastic system corresponding to the non-stationary forward stochastic system
\begin{subequations}\label{nonstat_csystforward}
\begin{align}
&dx=A(t)x(t)dt+B(t)dw,\quad x(0)=x_0\label{nonstat_csystforwarda}\\
&dy=C(t)x(t)dt+D(t)dw, \quad y(0)= 0  \label{nonstat_csystforwardb}
\end{align}
\end{subequations}
defined on the finite interval $[0,T]$, where $x_0$ (with covariance $P_0$) and the normalized Wiener process $w$ are uncorrelated.  To this end, apply the transformation 
\begin{equation}\label{Pbardefn}
\bar{x}(t)=P(t)^{-1}x(t)
\end{equation}
together with \eqref{xbarsystem_nonstationarycontb} to \eqref{nonstat_csystforwardb} to obtain 
\begin{displaymath}
dy=\bar{C}(t)\bar{x}(t)+D(t)d\bar{w}, 
\end{displaymath}
where 
\begin{equation}
\bar{C}(t)=C(t)P(t)+D(t)B(t).
\end{equation}
This together with \eqref{xbarsystem_nonstationaryconta} yields the the backward system corresponding to \eqref{nonstat_csystforward}, namely 
\begin{subequations}\label{nonstat_csystbackward}
\begin{align}
&d\bar{x}=-A(t)'\bar{x}(t)dt+\bar{B}(t)d\bar{w}  \label{nonstat_csystbackwarda}\\
&dy=\bar{C}(t)\bar{x}(t)dt+D(t)d\bar{w}. \label{nonstat_csystbackwardb}
\end{align}
\end{subequations}
with end-point condition $\bar{x}(T)=P(T)^{-1}x(T)$ uncorelated to the Wiener process $\bar{w}$. 

The backward realization \eqref{nonstat_csystbackward} was derived in \cite{Badawi-L-P-79}, but in cumbersome way, requiring the proof that $\bar{w}(t)$ is a normalized process with orthogonal increments to be suppressed. What is new here is imposing the unitary map between $w$ and $\bar{w}$, making the analysis much simpler and more natural. 

\section{Kalman filtering with missing observations}\label{Kalmansec}

We consider the linear stochastic system \eqref{nonstat_csystforward} which does not have a purely deterministic component that enables exact estimation of components of $x$ from $y$, an assumption that we retain in the rest of the paper. 
In the engineering literature is often the case that the stochastic system \eqref{nonstat_csystforward} represented as 
\begin{subequations}\label{eq:whitenoiserepr}
\begin{align}
&\dot x(t)=A(t)x(t)+B(t)\dot w(t),\quad x(0)=x_0\label{eq:whitenoiserepra}\\
&\dot y(t)=C(t)x(t)+D(t)\dot w(t) \label{eq:whitenoisereprb}
\end{align}
\end{subequations}
where the formal ``derivative'' $\dot w$ is white noise, i.e., $\E\{\dot w(t)\dot w(s)'\}=I\delta(t-s)$ with $\delta(t-s)$ being the Dirac ``function''. Of course $\dot x$, $\dot y$ and $\dot w$ are to be interpreted as generalized stochastic processes. From a mathematically rigorous point of view, observing $\dot y$ makes little sense since, for any fixed $t$, $\dot y(t)$ has infinite variance and contains no information about the state process $x$. However, observations of $\dot y$ could be interpreted as observations of the increments $dy$ of $y$ in a precise meaning to be defined next. 
%In this context, observations of $\dot y$ correspond to observations of the increments $dy$ of $y$. 
On the other hand, one can think of \eqref{nonstat_csystforward} as a system of type
\begin{displaymath}
dz=M(t)z(t)dt+N(t)dw(t),\quad \text{where $z(t)=\begin{bmatrix} x(t)\\y(t)\end{bmatrix}$},
\end{displaymath}
and one would like to determine the optimal linear least-squares estimate of $x(t)$ given past observed values of $y$.  

Generally this distinction between observing $y$ or $dy$ is not important. However, when there is loss of information over an interval $(t_1,t_2)$, there are two different information patterns depending on whether  $dy$ or $y$ is observed. The difference consists in whether $\Delta y:=y(t_2)-y(t_1)$ is part of the observation record or not. These two cases will be dealt with separately in subsections below. In fact, the former, which is common in engineering applications, is provided as a simplified preliminary, whereas our main interest is in the latter. 
To this end, we first introduce some notation. 

Consider the stochastic system \eqref{nonstat_csystforward} on a finite interval $[0,T]$.  As before, let $\bH$ be the Hilbert space spanned by  $\{w_k(t)-w_k(s);\, s,t\in [0,T],\, k=1,2,\dots,m\}$, endowed with the inner product $\langle\lambda,\mu\rangle =\E\{\lambda\mu\}$. For any $\lambda\in\bH$ and any subspace $\bA$, let $\E^\bA$ denote the orthogonal projection of $\lambda$ onto $\bA$. 
We denote by $\bH_{[t_1,t_2]}(dy)$ the (closed) subspace generated by the  components of the increments of the observation process $y$ over the window $[t_1,t_2]$. In particular, we shall also use the notations $\bH_t^-(dy):=\bH_{[0,t]}(dy)$ and $\bH_t^+(dy):=\bH_{[t,T]}(dy)$.

 Suppose that the output process or its increments are available for observation only on some subintervals of $[0,T]$, namely  $\mathcal{I}_k$, $k=1,2,\dots,\nu$. Next we want to define $\bHo$ as the proper subspace of $\bH_{[0,T]}(dy)$  spanned by the observed data. In the case that only the increments $dy$ or, equivalently, the ``derivative'' $\dot y$ is observed, we simply define 
 \begin{displaymath}
\bHo:=\bH_{\mathcal{I}_1}(dy)\vee\bH_{\mathcal{I}_2}(dy)\vee\cdots\vee\bH_{\mathcal{I}_\nu}(dy),
\end{displaymath}
In the case that the process $y$ is observed, we need to expand $\bHo$ by adding the subspaces spanned by the increments $\Delta y$ over the complementary intervals without observation. In either case, we define 
\begin{equation}
\label{eq:bHom}
\bHom :=\bHo\cap\bH_t^-(dy)\quad\text{and}\quad \bHop :=\bHo\cap\bH_t^+(dy).  
\end{equation}
Then Kalman filtering with missing observations amounts to determining a recursion for $x_-$ where 
\begin{equation}
\label{xminusdefn}
a^\prime x_-(t)=\EbHom a^\prime x(t), \quad \text{for all $a\in\mathbb{R}^n$}.
\end{equation}

\subsection{Observing $dy$ only}\label{sec:dyobs}

When observations are available on the interval $[0,t_1]$, the Kalman filter on that interval is given by
\begin{subequations}\label{eq:kf}
\begin{align}\label{eq:kalman1}
dx_-&=A(t)x_-(t)dt+K_-(t)(dy(t)-C(t)x_-(t)dt)\\\label{eq:kalman2}
K_-&= (Q_-C^\prime +B D^\prime) R^{-1}\\\label{eq:kalman3}
\dot Q_-(t) &=AQ_-+Q_-A^\prime -K_-RK_-^\prime +BB^\prime
\end{align}
\end{subequations}
with $R(t)=D(t)D(t)^\prime$ and initial conditions  $x_-(0)=0$ and $Q(0)=P_0$. Here $Q_-(t)$ is the error covariance
\begin{equation}
\label{Qminus}
Q_-(t):=\E\{[x(t)-x_-(t)](x(t)-x_-(t)]'\},
\end{equation}
which, by the nondeterministic assumption, is positive definite for all $t$. 

Next suppose the observation process becomes unavailable over the interval $[t_1,t_2)\subset [0,T]$. Then the Kalman filter needs to be modified accordingly. In fact, for any $t\in[t_1,t_2)$, \eqref{xminusdefn} holds with the space of observations $\bHom:=\bH_{t_1}^-(dy)$, and consequently
\begin{displaymath}
a^\prime x_-(t)=\E^{\bH_{t_1}^-(dy)}a^\prime x(t) =a'\Phi(t,t_1)x_-(t_1).
\end{displaymath}
This corresponds to setting $K_-(t)=0$ in \eqref{eq:kf} on the interval $[t_1,t)$  so that
\begin{subequations}\label{eq:freeevKalman}
\begin{equation}
dx_-=A(t)x_-(t)dt
\end{equation}
with initial condition $x_-(t_1)$ given by \eqref{eq:kalman1}. The error covariance $Q_-$ is then  given by the Lyapunov equation 
\begin{equation}
\label{QminusLyapunov}
\dot Q_-(t) =AQ_-+Q_-A^\prime  +BB^\prime
\end{equation}
\end{subequations}
with initial the condition $Q_-(t_1)$ given by the value produced in the previous interval. 

Then suppose observations of $dy$ become available again on the interval $[t_2,t_3)$. Then, for any $t\in[t_2,t_3)$, we have
\begin{displaymath}
\bHop=\bH_{[0,t_1]}\vee\bH_{[t_2,t]},
\end{displaymath}
so the Kalman estimate is generated by \eqref{eq:kf}  but now with initial conditions $x_-(t_2)$ and $Q_-(t_2)$ being those computed in the previous step without observation. In the case there are more intervals, one proceeds similarly by alternating between filters \eqref{eq:kf} and \eqref{eq:freeevKalman} depending on whether increments $dy$ are available or not.

In an identical manner, a cascade of  backward Kalman filters generates a process $\bar x_+(t)$ based on the backward stochastic realization \eqref{nonstat_csystbackward} and the observation windows $[t,T]$. Assuming that there are observations in a final interval ending at $t=T$, on that interval the Kalman estimate 
\begin{equation}
\label{xbarplusdefn}
a'\bar{x}_+(t)=\EbHop a'\bar{x}(t), 
\end{equation} 
with initial observation space $\bHop:=\bH_{[t,T]}$, is generated by the backward Kalman filter 
\begin{subequations}\label{eq:bkf}
\begin{align}
d\bar{x}_+&=-A(t)'\bar{x}_+(t)dt\notag \\
&\phantom{xxxxxx}+\bar K_+(t)(dy(t)-\bar{C}(t)\bar{x}_+(t)dt)\label{eq:bkalman1}\\
\bar K_+&= -(\bar{Q}_+\bar{C}^\prime -\bar{B} D^\prime) R^{-1}\label{eq:bkalman2}\\
\dot{\bar{Q}}_+ &=-A^\prime\bar{Q}_+ -\bar{Q}_+A +\bar{K}_+R(t)\bar{K}_+(t)^\prime -\bar{B}\bar{B}^\prime \label{eq:bkalman3}
\end{align} 
\end{subequations}
and initial conditions $\bar{x}_+(T)=0$ and $\bar{Q}_+(T)=\bar{P}(T)$ for $\bar{x}_+$ and the error covariance \begin{equation}
\bar{Q}_+(t):=\E\{[ \bar{x}(t)-\bar{x}_+(t)][ \bar{x}(t)-\bar{x}_+(t)]^\prime\},
\end{equation}
which like $Q_-(t)$ is positive definite for all $t$. 
During periods of no observations of $dy$, we then set the gain $\bar K_+=0$.  This update is  obtained  from the backward time stochastic model \eqref{xbarsystem_nonstationarycont} in an identical manner to that of \eqref{eq:freeevKalman}. 

Consequently,  both the underlying process as well as the filter can run in either time-direction. This duality becomes essential in subsequent sections where we will be concerned with smoothing and interpolation.

\subsection{Observing $y$}\label{sec:dyobs}

Now consider the case that $y$, and note merely $dy$,  is available for observation on all intervals $\mathcal{I}_k$, $k=1,2,\dots,\nu$. Under this scenario and with a continuous-time process the dynamics of Kalman filtering become hybrid, requiring both continuous-time filtering when data is available as well as a discrete-time update across intervals where measurements are not available.

Then on the first interval $[0,t_1]$ the Kalman estimate \eqref{eq:kf} will still be valid. However, 
when $t$ reaches the endpoint $t_2$ of the interval of no information and an observation of $y$ is obtained again, the subspace of observed data becomes
\begin{displaymath}
\bHotwo =\bH_{t_1}^-\vee \bH(\Delta y),
\end{displaymath}
where $\Delta y:=y(t_2)-y(t_1)$. 
Computing $x(t_2)$ across the window $(t_1,t_2]$ as a function of $x(t_1)$ and the noise components we have that
\begin{displaymath}
x(t_2)= \underbrace{\Phi(t_2,t_1)}_{A_d}x(t_1) + \underbrace{\int_{t_1}^{t_2}\Phi(t_2,s)Bdw(s)}_{u_1(t_1)}
\end{displaymath}
while
\begin{displaymath}
y(t_2)= y(t_1) + \int_{t_1}^{t_2} C(t) x(t) dt  + \int_{t_1}^{t_2} D(t)dw(t).
\end{displaymath}
Therefore,
\begin{displaymath}
\Delta y = \underbrace{\int_{t_1}^{t_2} C(t)\Phi(t,t_1)dt)}_{C_d} x(t_1) +u_2(t_1)
\end{displaymath}
where
\begin{displaymath}
\begin{split}
u_2(t_1)&=\int_{t_1}^{t_2}C(t)\int_t^{t_1}\Phi(t,s)B(s)dw(s)dt\\
&\phantom{xxxxxxxxxxxxxxxxxxxx}+ \int_{t_1}^{t_2} D(s)dw(s)\\
&=\int_{t_1}^{t_2}\underbrace{\left(\int_t^{t_2}C(t)\Phi(t,s)dt B(s)+ D(s)\right)}_{M(s)} dw(s).
\end{split}
\end{displaymath}
Thus, we obtain the discrete-time update
\begin{subequations}\label{eq:across}
\begin{align}
x(t_2)&=A_d x(t_1) + B_d v(t_1)\\
\Delta y&=C_d x(t_1) +D_d v(t_1)
\end{align}
\end{subequations}
where
\[
u(t_1)=\left(\begin{matrix}u_1(t_1)\\u_2(t_1)\end{matrix}\right)=
\left(\begin{matrix}B_d\\D_d\end{matrix}\right)v(t_1)
\]
and $B_d$ and $D_d$ are chosen so that
{\footnotesize
\[
\left(\begin{matrix}B_d\\D_d\end{matrix}\right)\hspace*{-3pt}\left(B_d^\prime,\;{D^\prime_d}\right)\hspace*{-1pt} = 
\hspace*{-1pt}\int_{t_1}^{t_2}\hspace*{-3pt}
\left(\begin{matrix}\Phi(t_2,s)BB^\prime \Phi(t,s) &\hspace*{-7pt} \Phi(t_2,s)BM(s)^\prime\\M(s)B^\prime\Phi(t_2,s)^\prime & M(s)M(s)^\prime
\end{matrix}\right)ds
\]
}
while $E\{v(t_1)v(t_1)^\prime\}=I$.

Hence, across the window of missing data the Kalman state estimate $x_-$ is now generated by a discrete-time Kalman-filter step
\begin{subequations}\label{eq:kf_gap}
\begin{align}\label{eq:kalman3}
x_-(t_2)&=A_d x_-(t_1)+K_d(\Delta y-C_d x_-(t_1))\\\label{eq:kalman4}
K_d&=(A_dQ(t_1)C_d^\prime+B_dD_d^\prime)\notag\\ 
&\phantom{xxxxxxxxxxx}\times(C_dQ(t_1)C_d^\prime + D_dD_d^\prime)^{-1}
\end{align}
with initial conditions $x_-(t_1)$ and $Q(t_1)$ given by \eqref{eq:kf} and the error covariance at $t_2$  by
\begin{align}\label{eq:kalman3}
Q(t_2) &= A_dQ(t_1)A_d^\prime -K_d(C_dQ(t_1)C_d^\prime \notag\\
&\phantom{xxxxxxxxxxxxxxxx} + D_dD_d^\prime)K_d'+B_dB_d^\prime.
\end{align}
\end{subequations}
In the next interval $[t_2,t_3]$, where observations of $y$  are available, the new Kalman estimate \eqref{xminusdefn} with 
\begin{displaymath}
\bHop=\bH_{[0,t_1]}\vee \bH(\Delta y)\vee\bH_{[t_2,t]}
\end{displaymath}
is again generated by the continuous-time Kalman filter \eqref{eq:kf} starting from  $x_-(t_2)$ and $Q(t_2)$ given by \eqref{eq:kf_gap}.

Again given an observation pattern, where intermittently  $y$ becomes unavailable for observation, the Kalman estimate \eqref{xminusdefn} can be generated in precisely this manner by a cascade of continuous and discrete-time Kalman filters.

\begin{remark}
The observation pattern of a continuous-time stochastic model, where $y$ becomes unavailable over particular time-windows, is closely related to hybrid stochastic models where continuous-time diffusion is punctuated by discrete-time transitions. Indeed, unless interpolation of the statistics within windows of unavailable data is the goal, the end points of such intervals can be identified and the same hybrid model utilized to capture the dynamics.
\end{remark}

\begin{remark}
A common engineering scenario is the case where the signal is lost while the observation noise is still present. This amounts to having $C\equiv 0$ over the corresponding window, and the Kalman estimates are obtained by merely running the filters \eqref{eq:kf} and \eqref{eq:bkf} in the two time directions with the modified condition on $C$. 
This situation does not cover the information patterns discussed above since, whenever $BD'\ne 0$, the Kalman gains do not vanish and information about the state process is available even when $C$ is zero.
\end{remark}

\subsection{Smoothing}

Given these intermittent forward and backward Kalman estimates, we shall derive a formula for the smoothing estimate 
\begin{equation}
\label{smoothingestimate}
a'\hat{x}(t):=\E^\bHo a'x(t),\quad a\in\mathbb{R}^n ,
\end{equation}
valid for both the cases discussed above, where 
\begin{equation}
\label{Hring}
\bHo:=\bHomn_t\vee\bHopn_t\subset\bH_{[0,T]}(dy)
\end{equation}
is the complete subspace of observations. This is discussed next.

\section{Geometry of fusion}\label{sec:geometryoffusion}

Consider the system \eqref{nonstat_csystforward}, and let $\bX(t)$ be the (finite-dimensional) subspace in $\bH$ spanned by the components of the stochastic state vector $x(t)$. Then it can be shown \cite[Chapter 7]{LPbook} that $\bH_{[0,t]}(dy)\perp\bH_{[t,T]}(dy)\mid \bX_t$, where $\bA\perp\bB\mid \bX$ denotes the conditional orthogonality
\begin{equation}
\label{splitting}
\langle \alpha -\E^\bX\alpha,\beta-\E^\bX\beta\rangle=0 \quad \text{for all $\alpha\in\bA$, $\beta\in\bB$}.
\end{equation}
Next, let $\bX_-(t)$ and $\bX_+(t)$ be the subspaces spanned by the components of the (intermittent) Kalman estimates $x_-(t)$ and $\bar{x}_+(t)$, respectively. Then since $\bX_-(t)\subset\bHomn_t\subset\bH_{[0,t]}(dy)$ and $\bX_+(t)\subset\bHopn_t\subset\bH_{[t,T]}(dy)$, we have 
\begin{displaymath}
\bX_-(t)\perp\bX_+(t)\mid\bX(t),
\end{displaymath} 
which is equivalent to 
\begin{subequations}\label{eq:splittingcommutative}
\begin{equation}
\label{splitting2}
\E^{\bX_+(t)}a'x_-(t)=\E^{\bX_+(t)}\E^{\bX(t)}a'x_-(t),\quad a\in\mathbb{R}^n 
\end{equation}
\cite[Proposition 2.4.2]{LPbook}. Therefore  the diagram
\begin{equation}
\label{commutative}
\alignedat5
{\bX_-\hspace*{3pt}}&{}&\overset{\hspace*{1pt}\E^{\bX_+}|_{\bX_-}}\longrightarrow &{} &\hspace*{-7pt}{\bX_+}\hspace*{7pt}\\
\;_{\E^{\bX}|_{\bX_-}}&\hspace*{-7pt}\searrow&{}&\nearrow&\hspace*{-3pt}\;_{\E^{\bX_+}|_{\bX}}\\
{}&{}&\bX\hspace*{10pt}&{}
\endalignedat
\end{equation}
\end{subequations}
commutes, where the argument $t$ has been suppressed.  

\begin{lemma}\label{Plemma}
Let $x(t)$, $\bar{x}(t)$, $x_-(t)$ and $\bar{x}_+(t)$ be defined as above. 
Then, for each $t\in[0,T]$,
\begin{itemize}
  \item[{\rm (i)}] $\E\{x(t)x_-(t)'\}=P_-(t)$
  \item[{\rm (ii)}] $\E\{\bar{x}(t)\bar x_+(t)'\}=\bar P_+(t)$
  \item[{\rm (iii)}] $\E\{\bar{x}_+(t)x_-(t)'\}=\bar{P}_+(t)P_-(t)$,
\end{itemize}
where $P_-(t):=\E\{x_-(t)x_-(t)'\}$ is the state covariance of the Kalman estimate $x_-(t)$ and $P_+(t):=\E\{\bar x_+(t)\bar x_+(t)'\}$ is the covariance of the backward Kalman estimate $\bar{x}_+(t)$.
\end{lemma}

\begin{proof}
By the definition of the Kalman filter, \eqref{xminusdefn} holds, 
and consequently the components of the estimation error $x(t)-x_-(t)$ are orthogonal to $\bH_t^-$ and hence to the components of $x_-(t)$. 
Therefore, 
\begin{displaymath}
\E\{x(t)x_-(t)'\}=\E\{x_-(t)x_-(t)'\}=P_-(t),
\end{displaymath} 
proving condition (i). Condition (ii) follows from a symmetric argument. To prove (iii) we use condition \eqref{eq:splittingcommutative}. To this end, first note that, by the usual projection formula,
\begin{equation}
\label{Hankel}
\begin{split}
\E^{\bX_+(t)}a'x_-(t)&=\E\{a'x_-(t)\bar{x}_+(t)\}\bar{P}_+(t)^{-1}\bar{x}_+(t)\\
&=a'\E\{x_-(t)\bar{x}_+(t)'\}x_+(t),
\end{split}
\end{equation}
where  $x_+(t):=\bar{P}_+(t)^{-1}\bar{x}_+(t)$ is the dual basis in $\bX_+(t)$ such that $\E\{x_+(t)\bar{x}_+(t)'\}=I$.
Moreover, 
\begin{displaymath}
\begin{split}
\E^{\bX(t)}a'x_-(t)&=E\{a'x_-(t)x(t)'\}P(t)^{-1}x(t)\\
&=a'\E\{x_-(t)x(t)'\}\bar{x}(t)=a'P_-(t)\bar{x}(t), 
\end{split}
\end{displaymath}
where we have used condition (i) and \eqref{Pbardefn}.
Next, set $b:=P_-a$ and form
\begin{displaymath}
\begin{split}
\E^{\bX_+(t)}b'\bar{x}(t)&=\E\{b'\bar{x}(t)\bar{x}_+(t)\}\bar{P}_+(t)^{-1}\bar{x}_+(t)\\
&=b'\E\{\bar{x}(t)\bar{x}_+(t)\}x_+(t)\\
&=b'\bar{P}_+(t)x_+(t),
\end{split}
\end{displaymath}
by condition (ii), and consequently
\begin{equation}
\label{CstarO}
\E^{\bX_+(t)}E^{\bX(t)}a'x_-(t)=a'P_-(t)\bar{P}_+(t)x_+(t). 
\end{equation}
Then condition (iii) follows from \eqref{splitting2}, \eqref{Hankel} and \eqref{CstarO}.
\end{proof}

\begin{remark}
The proof of condition (iii) in  Lemma~\ref{Plemma} could be simplified if $\bar{x}_+$ were a regular backward Kalman estimate without intermittent loss of information. In this case, $x_+=\bar{P}_+^{-1}\bar{x}_+$ would be generated by a forward stochastic realization belonging to the same class as \eqref{nonstat_csystforward} and $\E\{\bar{x}_+(t)x_-(t)'\}=\bar{P}_+(t)\E\{x_+(t)x_-(t)\}=\bar{P}_+(t)\E\{x_-(t)x_-(t)\}$.
\end{remark}

\begin{lemma}\label{framespacelemma}
For each $t\in [0,T]$, the smoothing estimate $\hat{x}(t)$, defined by \eqref{smoothingestimate},  is given by
\begin{equation}
\label{smoothingestimate2}
a'\hat{x}(t)=E^{\bH_t^\square} a'x(t),\quad a\in\mathbb{R}^n ,
\end{equation}
where  $\bH_t^\square$ is the subspace 
\begin{equation}
\label{framespace}
\bH_t^\square=\bX_-(t)\vee\bX_+(t).
\end{equation}
\end{lemma}

\begin{proof}
Following \cite{Lindquist-P-91,Badawi-L-P-79,LPbook}, define $\bN^-(t):=\bHomn_t\ominus\bX_-(t)$ and $\bN^+(t):=\bHopn_t\ominus\bX_+(t)$. Then 
\begin{displaymath}
\bHo=\bN^-(t)\oplus\bH_t^\square\oplus\bN^+(t).
\end{displaymath}
Now,  $a'(x(t)-x_-(t))$ is orthogonal to $\bHomn_t$ and hence to $\bN^-(t)$. Also $a'x_-(t)\perp\bN^-(t)$. Hence  $a'x(t)\perp\bN^-(t)$ as well. In the same way we see that $a'x(t)\perp\bN^+(t)$. Therefore \eqref{smoothingestimate2} follows. 
\end{proof}

Consequently, the information from the two Kalman filters can be fused into the smoothing estimate  
\begin{equation}
\label{xhatlinercomb}
\hat{x}(t)=L_-(t)x_-(t)+\bar{L}_+(t)\bar{x}_+(t)
\end{equation}
for some matrix functions $L_-$ and $\bar{L}_+$. 

\section{Universal two-filter formula}\label{sec:twofilters}

To obtain a robust and particularly simple smoothing formula that works also with an intermittent observation pattern, we assume that the stochastic system \eqref{nonstat_csystforward} has already been transformed via \eqref{eq:AB2FG_cont} so that, for all $t\in [0,T]$, 
\begin{equation}
\label{eq:x=barx}
x(t)=\bar{x}(t)
\end{equation}
and therefore
\begin{equation}
\label{ }
P(t)=\E\{x(t)x(t)'\}=I=\bar{P}(t). 
\end{equation}
Then the error covariances in the filtering formulas of Section~\ref{Kalmansec} are  
\begin{equation}
\label{QminusQplusbar}
Q_-=I-P_-\quad \text{and} \quad\bar{Q}_+=I-\bar{P}_+. 
\end{equation}
Consequently, $x(t)$, $\bar{x}(t)$, $P_-(t)$ and $\bar{P}_+(t)$ are all bounded in norm by one for all $t\in [0,T]$. 

\begin{theorem}
Suppose that \eqref{eq:x=barx} holds. 
For every $t\in [0,T]$, we have the formula 
\begin{equation}
\label{twofilterformula}
\hat{x}(t)=Q(t)\left(Q_-(t)^{-1}x_-(t)+\bar{Q}_+(t)^{-1}\bar{x}_+(t)\right)
\end{equation}
for the smoothing estimate \eqref{smoothingestimate}, where the estimation error 
\begin{equation}
\label{errorQ}
Q(t):=\E\left\{\left(x(t)-\hat{x}(t)\right)\left(x(t)-\hat{x}(t)\right)'\right\}
\end{equation}
is given by 
\begin{equation}
\label{Qformula}
Q(t)^{-1}=Q_-(t)^{-1}+\bar{Q}_+(t)^{-1}-I,
\end{equation}
and where $x_-$, $\bar{x}_+$, $Q_-$ and $Q_+$ are given by \eqref{eq:kf} and \eqref{eq:bkf} with boundary conditions $x_-(0)=\bar{x}_+(T)=0$ and $Q_-(0)=Q_+(T)=I$.
\end{theorem}

\begin{proof}
Clearly the matrix functions $L_-$ and $\bar{L}_+$ in \eqref{xhatlinercomb} can be 
determined from the orthogonality relations
\begin{subequations}\label{orthogonal}
\begin{equation}
\label{orthogonal1}
\E\{[x(t)-\hat{x}(t)]x_-(t)'\}=0
\end{equation} 
and
\begin{equation}
\label{orthogonal2}
\E\{[x(t)-\hat{x}(t)]\bar{x}_+(t)'\}=0. 
\end{equation}
\end{subequations}
By Lemma~\ref{Plemma}, \eqref{orthogonal} yields 
\begin{align*}
   &P_- - L_-P_-- \bar{L}_+\bar{P}_+P_- =0  \\
    &\bar{P}_+ - L_-P_-\bar{P}_+ -\bar{L}_+\bar{P}_+=0, 
\end{align*}
which, in view of the fact that $P_-$ and $\bar{P}_+$ are positive definite, yields
\begin{subequations}\label{LminusLplusbar}
\begin{equation}
 L_- +\bar{L}_+\bar{P}_+ =I
\end{equation}
\begin{equation}
L_-P_- +\bar{L}_+=I
\end{equation}
\end{subequations}
Again by orthogonality and Lemma~\ref{Plemma},
\begin{displaymath}
\begin{split}
Q=\E\left\{\left(x-\hat{x}\right)x'\right\} =I-L_-P_--\bar{L}_+\bar{P}_+,
\end{split}
\end{displaymath}
which, in view of \eqref{LminusLplusbar} and the relations \eqref{QminusQplusbar}, yields
\begin{equation}
\label{Q2LminusLplus}
L_-=QQ_-^{-1}\quad \text{and}\quad \bar{L}_+=Q\bar{Q}_+^{-1}.
\end{equation}
Then \eqref{twofilterformula} follows from \eqref{xhatlinercomb} and  \eqref{Q2LminusLplus}.
To prove \eqref{Qformula} eliminate $\bar{L}_+$ in \eqref{LminusLplusbar} to obtain
\begin{displaymath}
L_-(I-P_-\bar{P}_+)=\bar{Q}_+,
\end{displaymath}
which together with \eqref{Q2LminusLplus} yields
\begin{displaymath}
Q^{-1}=Q_-^{-1}(I-P_-\bar{P}_+)\bar{Q}_+^{-1}.
\end{displaymath}
However, 
\begin{displaymath}
I-P_-\bar{P}_+=\bar{Q}_+ +Q_- - Q_-\bar{Q}_+,
\end{displaymath}
and hence  \eqref{Qformula} follows. 
\end{proof}

In the special case with no loss of observation this is a normalized version of the Mayne-Frazer two-filter formula \cite{Mayne-66,Fraser-P-69}, which however in \cite{Mayne-66,Fraser-P-69} was formulated in terms of $x_-$ and $x_+$ rather than $\bar{x}_+$, where $x_+$ is the state process of the forward stochastic system of the backward Kalman filter. (For the corresponding formula in terms of $x_-$ and  $\bar{x}_+$, see \cite{Badawi-L-P-79,LPbook}; also cf. \cite{wall1981fixed}, where an independent derivation was given.) With a single interval of loss of observation the formula \eqref{twofilterformula} reduces  to a version of the interpolation formulas in \cite{Pavon-84a}. The remarkable fact, discovered here, is that the same formula \eqref{twofilterformula} holds for any intermittent observations structure and by a cascade of continuous and discrete-time forward and backward Kalman filters, as needed depending on the assumed information pattern. 

\section{Recap of computational steps}\label{sec:recap}

Given a system \eqref{nonstat_csystforward} with state covariance \eqref{eq:Lyapunovdifferentialeq}, make the normalizing substitution
\begin{equation}
\label{substitution}
\begin{split}
&A(t)\leftarrow P(t)^{-\half} A(t) P(t)^{\half}+R(t)\\
&B(t) \leftarrow P(t)^{-\half} B(t)\\
&C(t)\leftarrow C(t)P(t)^{\half} 
\end{split}
\end{equation}
with $R(t)=\left[\frac{d\phantom{t}}{dt}P(t)^{-\half}\right]P(t)^{\half}$. Next, we compute the intermittent forward and backward Kalman filter estimates $x_-$ and $\bar{x}_+$, respectively,  along the lines of Section~\ref{Kalmansec}, where, due to the normalization, $Q_-(0)=\bar{Q}_+(T)=I_n$. Then the smoothing estimate is given by 
\begin{displaymath}
\hat{x}(t)=Q(t)\left(Q_-(t)^{-1}x_-(t)+\bar{Q}_+(t)^{-1}\bar{x}_+(t)\right),
\end{displaymath}
where 
\begin{displaymath}
Q(t)=\left(Q_-(t)^{-1}+\bar{Q}_+(t)^{-1}-I\right)^{-1}.
\end{displaymath}

\section{An example}\label{sec:example}

We now illustrate the results of the paper on a specific numerical example.
We consider the continuous-time diffusion process
\begin{eqnarray*}
dx_1(t) &=& x_2(t)dt\\
dx_2(t) &=& -0.3x_1(t)dt -0.7x_2(t)dt +dw(t)\\
dy(t)&=& x_1(t)dt + dv(t)
\end{eqnarray*}
where $w$ and $v$ are thought to be independent standard Wiener processes. Here, $x_1$ is thought of as position and $x_2$ as velocity of a particle that is steered by stochastic excitation in $dw$, in the presence of a restoring force $0.3x_1$ and frictional force $0.7x_2$. Then $dy/dt$ represents measurement of the position and $dv/dt$ represents measurement noise (white).

\begin{figure}[h]\begin{center}
\includegraphics[width=0.45\textwidth]{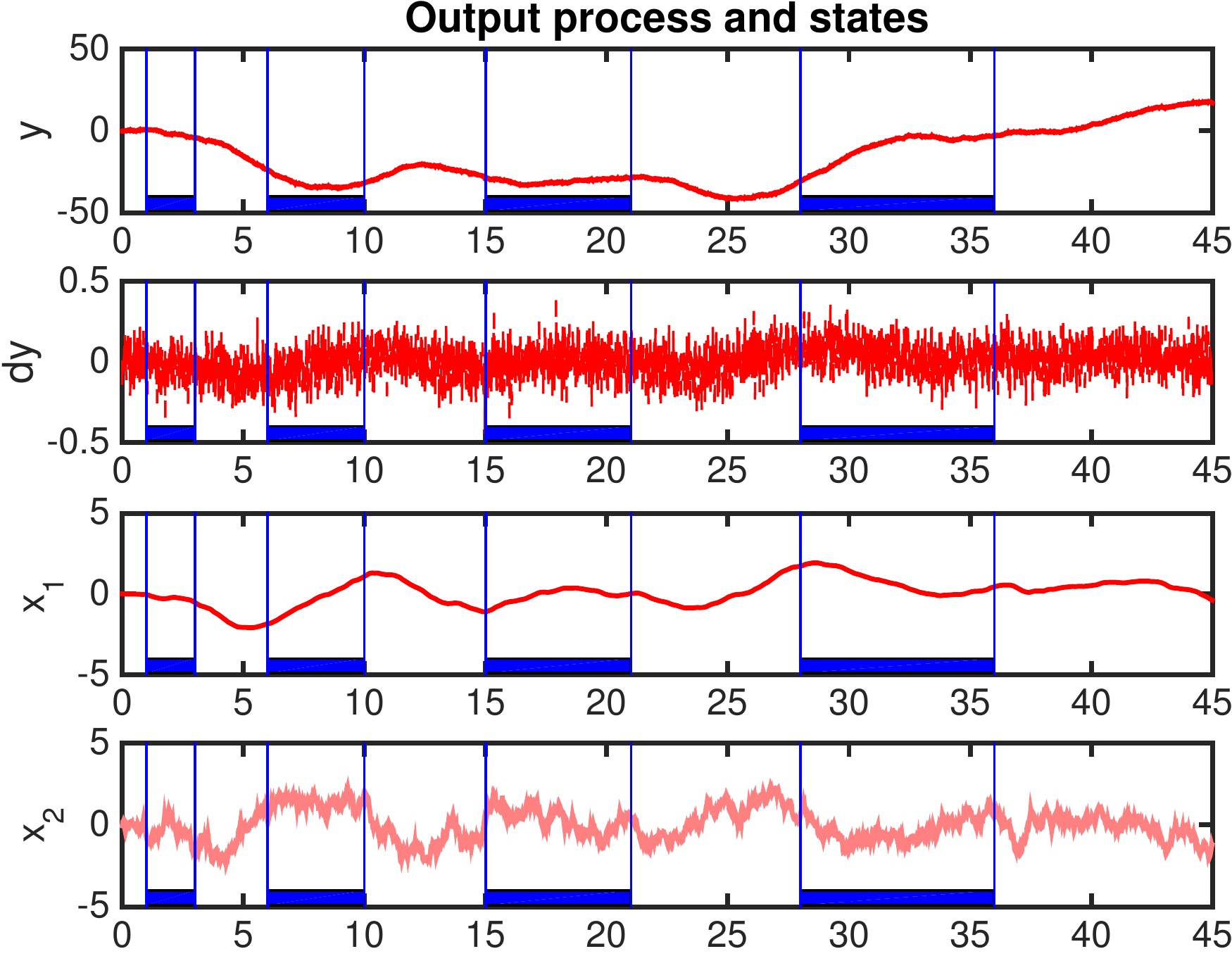}
   \caption{Sample paths of output process, increment, and state processes}
   \label{fig1}
\end{center}\end{figure}
\begin{figure}[h]\begin{center}
\includegraphics[width=0.45\textwidth]{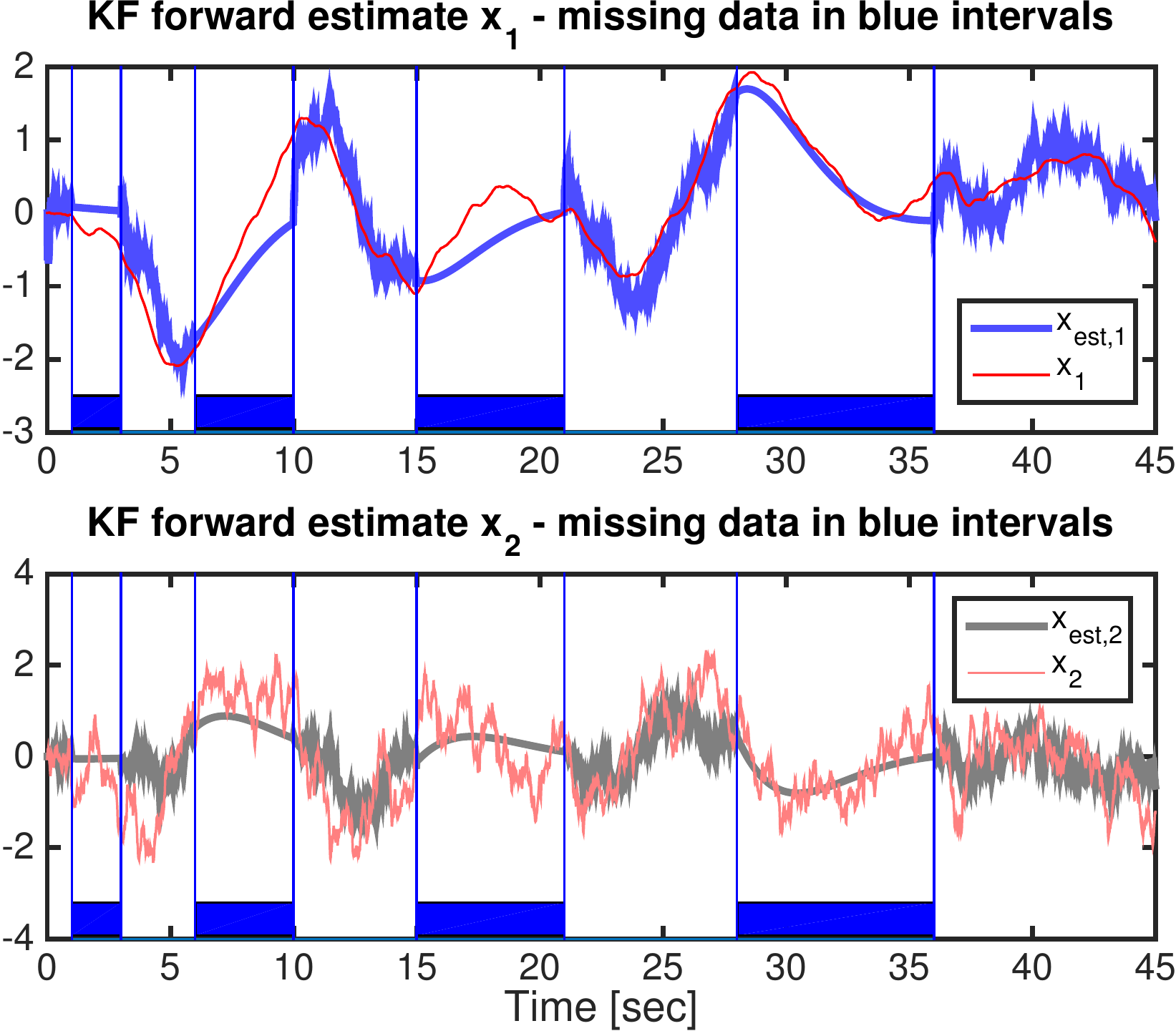}
   \caption{Kalman estimates in the forward time direction}
   \label{fig3}
\end{center}\end{figure}
Numerical simulation over $[0,T]$ with $T=45$ (units of time) produces a time-function $y(t)$ which is sampled with integer multiples of $\Delta t = 0.01$ (units). The interval $[0,T]$ is partitioned into
\[
[0,T]=\cup_{i=1}^9 [t_{i-1},t_i]
\]
where $t_0=0$ and $t_i-t_{i-1}=i$ (units). Measurements of $y$ are made available for purposes of state estimation over the intervals $[t_{i-1},t_i]$ for $i=1,3,5,9$. Over the complement set of intervals, data are not made available for state estimation; these intervals where data are not to be used are marked by a thick blue baseline in the figures. In Figure \ref{fig1} we display sample paths of the output process $y$, increments $dy$, and state-processes $x_1$ and $x_2$.

\begin{figure}[ht]\begin{center}
\includegraphics[width=0.45\textwidth]{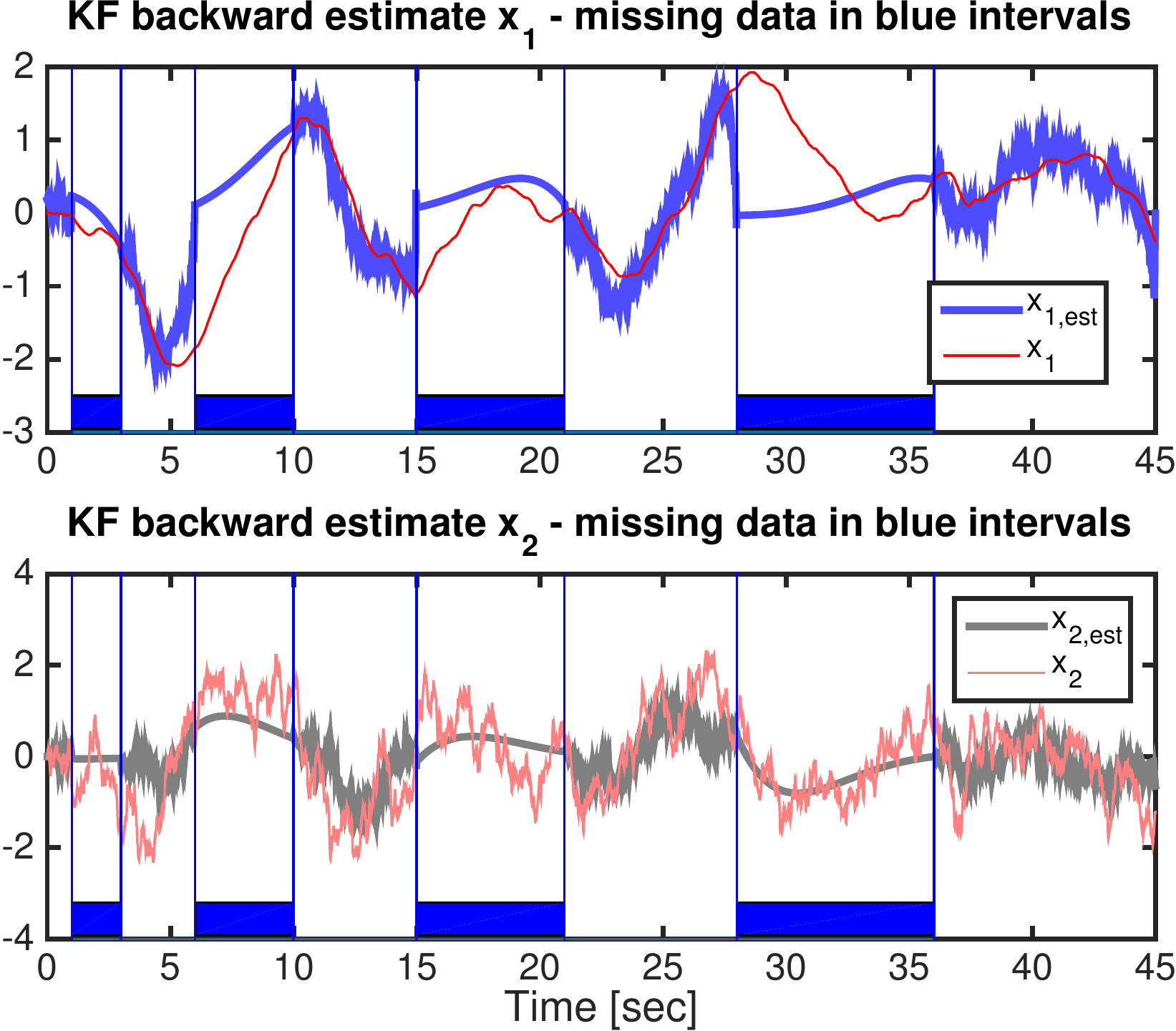}
   \caption{Kalman estimates in the backward time direction}
   \label{fig4}
\end{center}\end{figure}
\begin{figure}[h]\begin{center}
\includegraphics[width=0.45\textwidth]{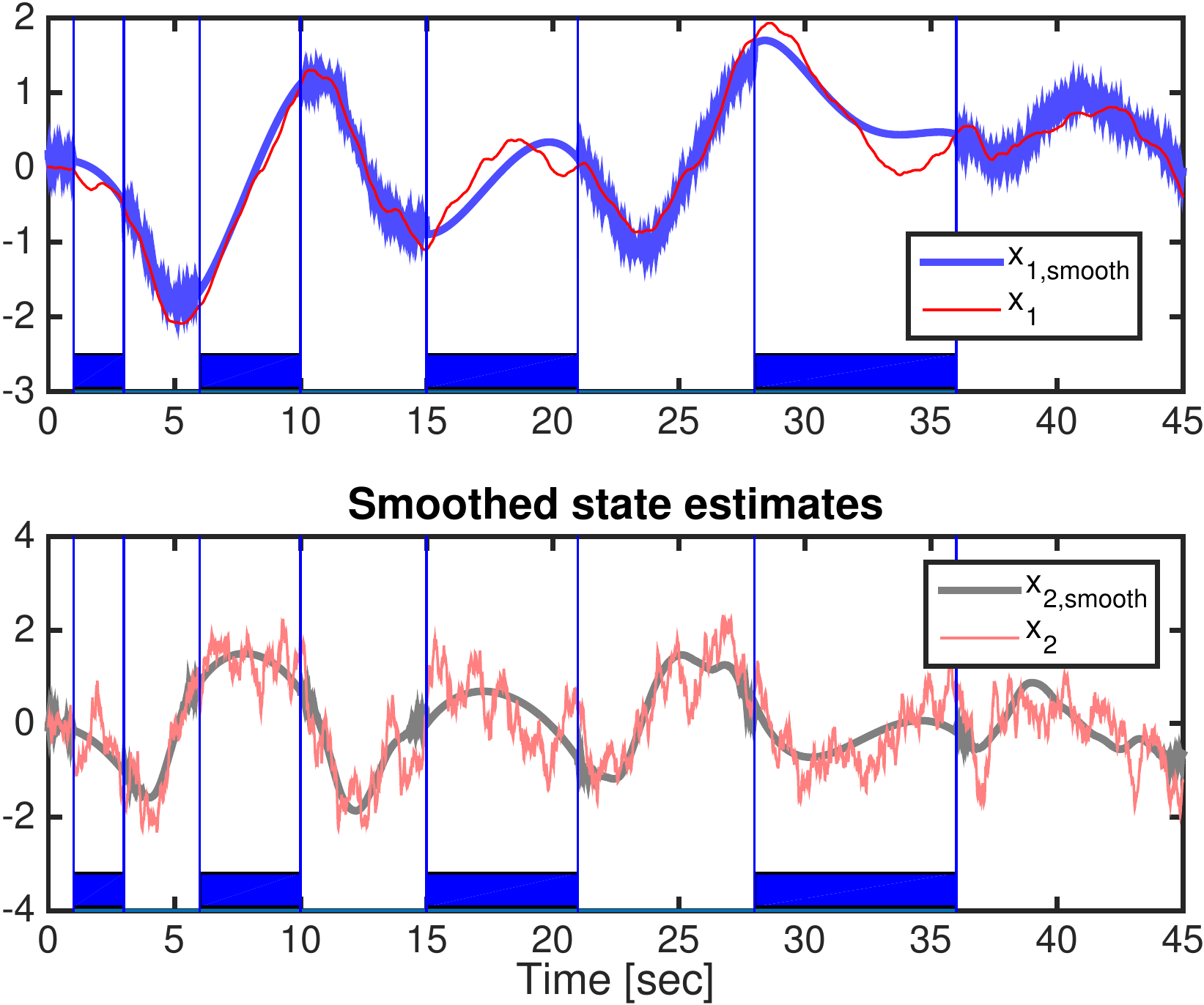}
   \caption{Interpolation/smoothed estimates by fusion of Kalman forward and backward estimates}
   \label{fig5}
\end{center}\end{figure}

The process increments $dy$ over $[t_{i-1},t_i]$ for $i=1,3,5,9$ as well as the increments $\Delta y$ across the $[t_{i-1},t_i]$ for $i=2,4,6,8$ are used in the two-filter formula for the purpose of smoothing. The Kalman estimates for the states in the forward and backwards in time directions, $x_-(t)$ and $\bar x_+(t)$
are shown in Figures \ref{fig3} and \ref{fig4}, respectively. The fusion of the two using \eqref{twofilterformula} is shown in Figure \ref{fig5}. It is worth observing the nature and fidelity of the estimates. In the forward direction, across intervals where data is not available, $x_-$ becomes increasing more unreliable whereas the opposite is true for $\bar x_+$, as expected.
The smoothing estimate is generally an improvement to those of the two Kalman filters as seen in Figure~\ref{fig5}.
In particular, it is worth noting  $x_2$ (in subplot 2), where, over windows of available observations, estimates have considerably less variance in the middle of the interval where the weights ($Q(t)Q_-(t)^{-1}$ and $Q(t)\bar Q_+(t)^{-1}$) in \eqref{twofilterformula} are equalized, whereas sample paths become increasing rugged at the two ends where one of the two Kalman estimates has significantly higher variance, and the corresponding mixing coefficient becomes relatively smaller.

\section{Concluding remarks}\label{sec:conclusions}

Historically the problem of interpolation has been considered from the beginning of the study of stochastic processes
\cite{kolmogorov1978stationary,yaglom1949problems}. Early accounts and treatments were cumbersome and non-explicit
as the problem was considered difficult \cite{karhunen1952interpolation,rozanovstationary,masani1971review,dym2008gaussian}. In a manner that echoes the development of Kalman filtering, the problem became transparent and computable for ouput processes of linear stochastic systems \cite{Pavon-84,Pavon-84a,LPbook}.

This paper builds on developments in stochastic realization theory \cite{Lindquist-P-79,pavon1980stochastic} and presents a unified and generalized two-filter formula for smoothing and interpolation in continuous time  for the case of intermittent availability of data over an operating window. The analysis considers two alternative information patterns where increments of the output process or the output process itself is recorded when information becomes available. The second information pattern appears most natural to us in this continuous-time setting, and this is our main problem. Nevertheless, in either case, two Kalman filters run in opposite time-directions, designed on the basis of a forward and a backward model for the process, respectively. Fusion of the respective estimates is effected via linear mixing in a manner similar to the Mayne-Fraser formula and applies to both smoothing and interpolation intermixed. In earlier works, smoothing and interpolation have been considered separate problems \cite[Chapter 15]{LPbook}. The balancing normalization also simplifies the mixing formula and makes it completely time symmetric. 

The theory relies on time-reversal of stochastic models. We provide a new derivation of such a reversal which has the convenient property of being balanced. It is based on lossless imbedding of linear systems and effects the time reversal through a unitary transformation.
Interestingly, time symmetry in statistical and physical laws
have occupied some of the most prominent minds in science and mathematics.
%(\cite{schrodinger1931umkehrung}, \cite{kolmogorov1992selected}, \cite{shiryayev1992reversibility}). 
In particular, closer to our immediate interests, dual time-reversed models have been employed to model, in different time-directions, Brownian or Schr\"odinger bridges \cite{pavon1991free}, \cite{dai1990markov}, a subject which is related to reciprocal processes \cite{jamison1974reciprocal}, \cite{Krener-86}, \cite{levy1990modeling}, \cite{dai1991stochastic}.
%, and also considered in thermodynamics and systems theory \cite{haddad2008time,haddad2009thermodynamics}.
A natural extension of the present work in fact is in the direction of general reciprocal dynamics \cite{Krener-86,levy1990modeling} and the question of whether similar two-filter formula are possible.

\section*{Appendix: Time reversal of non-stationary discrete-time systems}

Next, instead of \eqref{eq:model_discrete}, consider the non-stationary state dynamics
\begin{align}\label{eq:model_discrete_nonstationary}
x(t+1)=A(t)x(t)+B(t)w(t), \quad x(0)=x_0, 
\end{align}
on a finite time-window $[0,T]$, where, for simplicity we now assume that the covariance matrix $P_0:=P(0)$ of the zero-mean stochastic vector $x_0$  is positive definite, i.e., $P_0=\E\{x_0x_0'\}>0$. 
Then the state covariance matrix $P(t):=\E\{x(t)x(t)^\prime\}$ will satisfy the Lyapunov difference equation
\begin{align}\label{discreteLyapunov_diff}
P(t+1)= A(t)P(t)A(t)^\prime + B(t)B(t)^\prime.
\end{align}
The state transformation 
\begin{align}\label{eq:statetransform_nonstationary}
\xi(t)= P(t)^{-\half} x(t)
\end{align}
brings the system \eqref{eq:model_discrete_nonstationary} into the form
\begin{align}\label{eq:model_discrete2_nonstationary}
\xi(t+1)=F(t)\xi(t)+G(t)w(t),
\end{align}
where now $\E\{\xi(t)\xi(t)^\prime\}=I_n$ for all $t$ and
\begin{subequations}\label{eq:AB2FG_discrete}
\begin{align}
F(t)&=P(t+1)^{-\half} A(t) P(t)^{\half},\\G(t)&=P(t+1)^{-\half} B .
\end{align}
\end{subequations}
The Lyapunov difference equation then reduces to 
\begin{align}
I_n=F(t)F(t)^\prime + G(t)G(t)^\prime
\end{align}
allowing us to embed $[F,G]$  as part of a time-varying orthogonal matrix
\begin{align}\label{eq:U(t)}
U(t)=\left[\begin{array}{cc} F(t)& G(t)\\ H(t)& J(t)\end{array}\right] .
\end{align}
This amounts to extending \eqref{eq:model_discrete2_nonstationary} to 
\begin{subequations}\label{xisystem_nonstationary}
\begin{align}
\xi(t+1)&=F(t)\xi(t)+G(t)w(t)\\
\bar w(t)&=H(t)\xi(t)+J(t)w(t),
\end{align}
\end{subequations}
or, in the equivalent form
%%% as this is the same as 
\begin{equation}\label{eq:xiudynamics}
\begin{bmatrix}\xi(t+1)\\\bar{w}(t)\end{bmatrix}=U(t)\begin{bmatrix}\xi(t)\\w(t)\end{bmatrix}. 
\end{equation}
Hence, since $\E\{\xi(t)\xi(t)^\prime\}=I_n$ and $\E\{w(t)w(t)^\prime\}=I_p$,
%%%
and assuming that $\E\left\{\xi(t)w(t)^\prime\right\}=0$,
\begin{equation}
\label{ }
\E\left\{\begin{bmatrix}\xi(t+1)\\\bar{w}(t)\end{bmatrix}\begin{bmatrix}\xi(t+1)\\\bar{w}(t)\end{bmatrix}^\prime\right\}= U(t)U(t)^\prime =I_{n+p}, 
\end{equation}
which yields 
\begin{subequations}\label{eq:backward_correlations}
\begin{align}
    &\E\{\xi(t+1) \bar{w}(t)^\prime\}=0,   \\
    &\E\{\bar{w}(t) \bar{u}(t)^\prime\}=I_p .
\end{align}
\end{subequations}
Moreover, from \eqref{xisystem_nonstationary} we have 
\begin{displaymath}
\begin{split}
&\bar{u}(t+k)=H(t+k)\Phi(t+k,t)\xi(t)\\&+\sum_{j=t}^{t+k-1}H(t+k)\Phi(t+k,j+1)G(j)w(j) +J(t)w(t)
\end{split}
\end{displaymath}
for $k>0$, where 
\begin{displaymath}
\Phi(s,t)=
\begin{cases}
F(s-1)F(s-2)\cdots F(t)\quad \text{for $s>t$}\\
I_n\quad \text{for $s=t$}.
\end{cases}
\end{displaymath}
Therefore, since $F(t)H(t)'+G(t)J(t)'=0$ by the unitarity of $U(t)$,
\begin{displaymath}
\begin{split}
&\E\{\bar{u}(t+k)\bar{u}(t)'\}\\& =H(t+k)\Phi(t+k,t+1)[F(t)H(t)'+G(t)J(t)']=0. 
\end{split}
\end{displaymath}
Consequently, $\bar{u}$ is a white noise process. Finally,
premultiplying \eqref{eq:xiudynamics} by $U(t)^\prime$, we then obtain
 \begin{subequations}\label{xisystem_nonstationary_backward}
\begin{align}
\xi(t)&=F(t)'\xi(t+1)+H(t)'\bar{w}(t)\\
w(t)&=G(t)'\xi(t+1)+J(t)'\bar{w}(t),
\end{align}
\end{subequations}
which, in view of \eqref{eq:backward_correlations}, is a backward stochastic system. 

Using the transformation \eqref{eq:statetransform_nonstationary}, \eqref{xisystem_nonstationary} yields the forward representation
\begin{subequations}\label{eq:forwardunitary}
\begin{align}
x(t+1)&=A(t)x(t)+B(t)w(t)\\
\bar w(t)&=\bar{B}(t)'x(t)+J(t)w(t),
\end{align}
\end{subequations}
where $\bar{B}(t):=P(t)^{-\half}H(t)'$. Likewise  \eqref{xisystem_nonstationary_backward} and
\begin{equation}
\label{eq:x2xbar}
\bar{x}(t)=P(t+1)^{-1}x(t+1),
\end{equation}
yields the backward representation
\begin{subequations}
\begin{align}\label{eq:inversexsystema}
\bar{x}(t-1)&=A(t)^\prime\bar{x}(t)+\bar B(t) \bar w(t)\\
w(t)&=B(t)^\prime \bar{x}(t)+J(t)^\prime \bar w(t) . \label{eq:inversexsystemb}
\end{align}
\end{subequations}

\begin{remark}
When considered on the doubly infinite time axis,
equation \eqref{eq:xiudynamics} defines an isometry. Indeed,
assuming that the input is squarely summable, the fact that $U(t)$ is unitary for all $t$ directly implies that
\[
\sum_{-\infty}^N\|\bar w\|^2 + \|\xi(t+1)\|^2 = \sum_{-\infty}^N\|w(t)\|^2.
\]
Then, $\xi(t)\to 0$ as $t\to \infty$, provided $\Phi(t,s)\to 0$ as $s\to -\infty$.
It follows that
\[\sum_{t=-\infty}^\infty\|\bar w(t)\|^2=\sum_{t=-\infty}^\infty\|w(t)\|^2.
\]
\end{remark}

We are now in a position to derive a backward version of a non-stationary stochastic system
\begin{subequations}\label{dsystforward_nonstationary}
\begin{align}
&x(t+1)=A(t)x(t)+B(t)w(t),\quad x(0)=x_0  \label{nonstat_dsystforwarda}\\
&\phantom{xxll}y(t)=C(t)x(t)+D(t)w(t) \label{nonstat_dsystforwardb}
\end{align}
\end{subequations}
where $x_0$  and the normalized white-noise process $w$ are uncorrelated and $\E\{x_0x_0'\}=P_0$. 
In fact, inserting the transformations  \eqref{eq:x2xbar} and \eqref{eq:inversexsystema} into \eqref{nonstat_dsystforwardb} yields
\begin{displaymath}
y(t)=\bar{C}\bar{x}(t) +\bar{D}\bar{w}(t),
\end{displaymath}
where
\begin{align}
 \bar{C}   & =C(t)P(t)A(t)'+D(t)B(t)' \\
  \bar{D}  & =C(t)P(t)\bar{B}(t) + D(t)J(t)' 
\end{align}
From that we have the backward system 
\begin{subequations}\label{nonstat_dsystbackward}
\begin{align}
&\bar{x}(t-1)=A(t)'\bar{x}(t)+\bar{B}(t)\bar{w}(t) \label{nonstat_dsystbackwarda}\\
&\phantom{xxx}y(t)=\bar{C}(t)\bar{x}(t)+ \bar{D}(t)\bar{w}(t) \label{nonstat_dsystbackwardb}
\end{align}
\end{subequations}
with the boundary condition $\bar{x}(T-1)=P(T)^{-1}x(T)$ being uncorrelated to the white-noise process $\bar{w}$.

\bibliographystyle{IEEEtran}
\bibliography{ifacconf}    

%\subsection{Hybrid stochastic systems}
%
%$$\bHom\quad \bH_t^- \qquad \bHop\quad \bH_t^+\qquad \bHotwo\quad \bHomn$$
%
%$$\EbHom\qquad \EbHop$$                      
\end{document}